\documentclass[11pt]{amsart}
\usepackage{amsmath}
\usepackage{amsthm}
\usepackage{graphicx}
\graphicspath{{Figures/}}
\usepackage{tikz}
\usepackage{enumerate}
\usepackage{mathtools}
\usepackage{bm}

\allowdisplaybreaks

\newcommand{\ab}{\allowbreak}

\newcommand{\E}{\mathrm{E}}
\newcommand{\tr}{\mathrm{tr}}
\newcommand{\Tr}{\mathrm{Tr}}

\newcommand{\cm}{\ols{\hbox{{\mycal m}}}}
\newcommand{\cN}{\mathcal N}
\newcommand{\cP}{\mathcal P}

\newcommand{\bZ}{\mathbb{Z}}
\newcommand{\tg}{{\skew{2.0}\tilde\gamma}}
\newcommand{\hg}{{\skew{3.0}\hat\gamma}}
\newcommand{\tom}{{\tilde\omega}}
\DeclareMathOperator{\cov}{\mathrm{cov}}

\def\mycal{\fontfamily{ppl}\fontseries{m}%
\fontshape{it}\fontsize{10.25}{12}\selectfont}

\def\bR{\mathbb{R}}

\newcommand{\ds}{\displaystyle}

\newcommand{\ols}[1]{\mskip.5\thinmuskip
\overline{\mskip-.5\thinmuskip {#1} \mskip-.5
\thinmuskip}\mskip.5\thinmuskip}

\newcommand{\xx}[2]{\relax}
\newcommand{\thebottomline}{\renewcommand{\thefootnote}{}
  \renewcommand{\footnoterule}{}
  \phantom{M}\footnotetext{\tiny{}\hfill
    \textit{\noindent\romannumeral\day.%
\romannumeral\month.\expandafter\xx\romannumeral\year}}}

\def\ybox{\dimen0=\hsize \dimen1=30pt%
\divide\dimen0 by \dimen1
\count100=\dimen0 \dimen0=\hsize
\divide\dimen0 by \count100%
\dimen1=\dimen0\divide\dimen1 by 2
{\parindent0pt\hbox to
\hsize{\cleaders\vbox{\hsize\dimen0
\hrule\vrule  height1.2pt%
\advance\dimen1 by -0.8pt\hskip
\dimen1\advance\dimen1 by 1.2pt \vrule
width\dimen1\hrule}\hfill}}}

\newtheorem{theorem}{Theorem}
\newtheorem{lemma}[theorem]{Lemma}
\newtheorem{proposition}[theorem]{Proposition}
\newtheorem{corollary}[theorem]{Corollary}

\theoremstyle{definition}
\newtheorem{definition}[theorem]{Definition}
\newtheorem{notation}[theorem]{Notation}

\newtheorem{remark}[theorem]{Remark}

\usepackage{mathtools}

\DeclarePairedDelimiter{\abs}{\lvert}{\rvert}
\DeclarePairedDelimiter\floor{\lfloor}{\rfloor}

\begin{document}

\title[Asymptotic Infinitesimal Law of a real Wishart Matrix]
{The Asymptotic Infinitesimal Distribution\\[5pt] of a Real Wishart Random Matrix}

\author[mingo]{James A. Mingo} \address{Department
  of Mathematics and Statistics, Queen's University, Jeffery
  Hall, Kingston, Ontario, K7L 3N6, Canada}

\email{mingo@mast.queensu.ca} 

\thanks{Research supported by a Discovery Grant from
  the Natural Sciences and Engineering Research Council of
  Canada}

\author[Vazquez-Becerra]{Josue Vazquez-Becerra}
\address{Universidad Aut\'onoma Metropolitana, Iztapalapa
  Campus, C. P. 09340, Iztapalapa D.F., Mexico}

\email{13jdvb1@queensu.ca}

\begin{abstract}
Let $X_N$ be a $N \times N$ real Wishart random matrix with
aspect ratio $M/N$. The limit eigenvalue distribution of
$X_N$ is the Marchenko-Pastur law with parameter $c = \lim_N
M/N$. The limit moments $\{m_n\}_n$ are given by $m_n =
\sum_{\pi} c^{\#(\pi)}$ where the sum runs over $NC(n)$.
  
Let $m_n'$ be the limit of $N( \E(\tr(X_N^n)) - m_n)$.
These are the asymptotic infinitesimal moments of a real
Wishart matrix. We show that $m'_n$ can be written as a
sum over planar diagrams with two terms, $\sum_{\pi}
c'(\#(\pi) -1) c^{\#(\pi)-1}$, and $\sum_{\pi \in
  S_{NC}^\delta(n, -n)} c^{\#(\pi)/2}$, where
$S_{NC}^\delta(n, -n)$ is a set of non-crossing annular
permutations satisfying a symmetry condition. Moreover we present a recursion formula for the second term which is related to one for higher order freeness. 
\end{abstract}

\maketitle

\section{introduction and statement of main results}
\label{section:introduction}

In 1974 G.~'t Hooft \cite{h} considered gauge theory with colour gauge group $U(N)$ and quarks having a colour index from $1$ to $N$. He showed that in the limit $N \rightarrow \infty$ ``only planar diagrams with quarks at the edges dominate''. 
In 1991 D.~Voiculescu \cite{v} showed that, also in the limit $N \rightarrow \infty$, free independence described the behaviour of independent and unitarily invariant matrix ensembles. In 1994 R.~Speicher \cite{s} tied these together when he showed that free independence could be described with planar diagrams. 

We introduce a new class of planar diagrams that lie between the non-crossing partitions that connect moments to free cumulants and the annular diagrams used by Nica and Mingo \cite{mn} to describe the global fluctuations of random matrices. We show that these diagrams describe the infinitesimal laws for some orthogonally invariant ensembles.

Given a random matrix ensemble $\{ X_N \}_{N=1}^\infty$
where $X_N$ is $N \times N$ self-adjoint random matrix, we
say the ensemble has a \textit{limit distribution} (or a
limit eigenvalue distribution) if the random probability
measure $\mu_N$ on $\bR$ which has a mass of weight $1/N$ at
each eigenvalue of $X_N$ converges, as $N \rightarrow
\infty$, to a non-random probability measure, $\mu$. We will
study the averaged moments $\E( \int t^n \, d\mu_N(t)) =
\E(\tr( X_N^n ))$. The existence of a limit law comes to
showing that for each $n \geq 1$ we have
\[
\lim_N \E( \tr( X_N^n )) = \int t^n \, d\mu(t). 
\]
We let $m_n(N) = \E( \tr( X_N^n ))$ and $m_n = \int t^n \,
d\mu(t)$. If we have as above that $m_n = \lim_N m_n(N)$,
then we can consider the difference quotient
\[
m'_n(N) = \frac{m_n(N) - m_n}{N^{-1} - 0}.
\]
If the limit $\lim_{N \rightarrow \infty} m_n'(N)$ exists
for each $n$ we say that the ensemble $\{ X_N\}_N$ has a
\textit{limit infinitesimal distribution} and denote the resulting
moments by $m_n' = \lim_N m'_n(N)$. The existence of
infinitesimal laws were demonstrated by Johansson \cite{j} (in the
Hermite case), Mingo and Nica \cite{mn} and Dumitriu and Edelman
\cite{dumitriu_edelman_2006} (in the Laguerre case).

In \cite{bs}, Belinschi Shlyakhtenko introduced infinitesimal freeness as stronger form of freeness for random variables. In \cite{fn} F\'evrier and Nica showed the connection of infinitesimal freeness to freeness of type B. In \cite{sh} Shlyakhtenko has shown how infinitesimal freeness can be used to analyze spike models in random matrix theory. This paper introduces a class of planar objects extending those in \cite{m} that can be used to analyze infinitesimal freeness when the matrices only have orthogonal invariance. 

For a real or complex Wishart matrix the limit distribution
is the Mar\-chen\-ko-Pastur law. This is a probability measure
$\mu_c$ on $[0, \infty)$, which depends on a shape parameter
  $c > 0$. The moments of $\mu_c$ are given by simple
  combinatorial formula
\[
m_n = \sum_{\mathclap{\pi \in NC(n)}}  c^{\#(\pi)}
\]
where $NC(n)$ is the set of non-crossing partitions of $[n]
= \{1, 2, 3, \dots, n\}$ and for a partition $\pi$,
$\#(\pi)$ denotes the number of blocks of $\pi$. This is
equivalent to showing that for each $n$, $\kappa_n = c$,
where $\kappa_n$ is the $n^{th}$ free cumulant of $\mu_c$,
see \cite[Def. 12.12 \& Ex. 22.18]{ns}.

If we suppose that $X_N$ is a complex Wishart random matrix,
$\lim_N M/N = c$, and in addition that $c' := \lim_N M - c
N$ exists then it shown in \cite[Cor. 9.4]{mn} that
\begin{equation}\label{eq:shape_term}
m'_n = \sum_{\mathclap{\pi \in NC(n)}} c' \#(\pi) c^{\#(\pi)-1}.
\end{equation}
This can also be expressed by saying that the infinitesimal
free cumulant $\kappa'_n = c'$ for all $n$. In addition in
\cite[Thm. 3.1]{m} a signed measure $\nu'_1$ depending on
$c$ and $c'$, was given so that $m'_n = \int t^n\,
d\nu'_1(t)$. In this paper we consider the case where $X_N$
is a real Wishart matrix. We find that when $\lim_N M/N =
c$, and $c' := \lim_N M - c N$ exists there is also a limit
infinitesimal distribution:
\begin{equation}\label{eq:two_terms}
m'_n = \sum_{\pi \in NC(n)} c' \#(\pi) c^{\#(\pi)-1}
+
\sum_{\pi \in S_{NC}^\delta(n, -n)} c^{\#(\pi)/2}
\end{equation}
where $S_{NC}^\delta(n, -n)$ is a new combinatorial object whose
study is one of the main objects of this paper. The first term,
let's call it the \textit{shape
  term}, does not appear in the tridiagonal model of
Dumitriu and Edelman \cite{dumitriu_edelman_2006}. The second term
is the same as was found in the tridiagonal model of Dumitriu 
and Edelman.

\begin{figure}
\begin{center}{\small
\begin{tabular}{l|ll}
$n$  & $m_n$ & $m'_n$ \\ \hline
1  & $c$ & $0$ \\
2  & $c + c^2$ & $c$ \\ 
3  & $c + 3 c^2 + c^3$ & $3 c + 3 c^2$ \\
4  & $c + 6c^2 + 6c^3 + c^4$ & $6 c + 17 c^2 + 6 c^3$ \\ 
5  & $c + 10c^2 + 20c^3 + 10c^4 + c^5$ & $10 c + 55 c^2 + 55 c^3 + 10 c^4$\\ 
6  & $c + 15 c^2 + 50 c^3 + 50 c^4 + 15 c^5 + c^4$ &  $15 c + 135 c^2 + 262 c^3 + 135 c^4 + 15 c^5$ \\ 
\end{tabular}}\end{center}
\caption{\small\label{table:the_first_six_moments}The first 6 moments and infinitesimal moments of the
Mar\-chenko-Pastur law, assuming $c' = 0$. The coefficient of $c^k$ in $m_n$ is $\frac{1}{n} \binom{n}{k-1} \binom{n}{k}$. The coefficient of $c^k$ in $m'_n$ is $\frac{1}{2} \big( \binom{2n}{2k} - \binom{n}{k}^2 \big)$. }
\end{figure}

If we look at the second term of equation (\ref{eq:two_terms}) 
and let 
\begin{equation}\label{eq:m_bar}
\ols{m}'_n = \sum_{\mathclap{\pi \in S_{NC}^\delta(n, -n)}}  c^{\#(\pi)/2}
\end{equation}
we can then show that these
moments $\{ \ols{m}'_n \}_n$ satisfy a recursion relation
\[
\ols{m}'_n  = (n-1)m_{n-1} + (1+c) \ols{m}'_{n-1} + 2\sum_{k=2}^{n-2} \ols{m}'_{k} m_{n-k-1}
\]
which is independent of $c'$ but depends on the moments $\{m_n \}_n$. See Figure \ref{table:the_first_six_moments} above for values of $m_n$  and $\ols{m}'_n$ for small values of $n$. 
In the case $c = 1$, this recursion relation has already appeared in
the work of Merlini, Sprugnoli, and Verri \cite{msv}  where $\ols{m}'_{n+1}$ is the sum of the areas under Dyck paths of length $2n$. The same recursion appeared in F\'eray \cite[Lemma 3.12]{f}, where $\ols{m}'_n$ appears in the $1/N$ expansion of the orthogonal Weingarten function as a coefficient of a subleading term. There is a simple relation connecting $S_{NC}^\delta(n, -n)$ to the orthogonal Weingarten calculus, but we prefer to leave this for a later discussion. We obtain our recursion from a simple `pinching' construction that is the same as that which gives the recursion for Catalan numbers; see Figure \ref{fig10:intermediate_pinch}.

One should also compare the recursion above with the one above Theorem 5.24 in \cite{cmss}
\begin{align*}
&(-1)\mu(1_{m+n},\gamma_{m,n})= m \cdot\mu(1_{m+n},\gamma_{m+n})\\
&\quad+\sum_{\mathclap{1\leq k\leq n-1}} 
\big\{ \mu(1_{m+k},\gamma_{m,k})
\mu(1_{n-k},\gamma_{n-k}) + \mu(1_{m+n-k},\gamma_{m,n-k}) \mu(1_k,
\gamma_{k}) \big\},
\end{align*}
where $\mu(1_n, \gamma_n)$ is the M\"obius function for one cycle with $n$ points and $\mu(1_{m+n},\gamma_{m+n})$ is the M\"obius function for $2$ cycles, one with $m$ points and one with $n$ points. Recently higher order generalizations have been found by Borot, Charbonnier, Garcia-Failde, Leid, and Shadrin in \cite{bcgls}.

We show that the Cauchy transform of $\{ m'_n\}_n$ is 
\begin{align*}\lefteqn{
g(z) = \sum_{n=0}^\infty \frac{m'_n}{z^{n+1}} } \\
& =
\frac{-c'}{z \sqrt{P(z)}}
\frac{(1 - c)^2 - (1 + c)z - (1 - c)\sqrt{P(z)}}{\sqrt{P(z)} + z -1 + c} \\
& \qquad\mbox{} + 
\frac{1}{2}
\Big\{\ 
\frac{1}{2}\Big\{ \frac{1}{z - a} + \frac{1}{z -b}\Big\}
-
\frac{1}{\sqrt{(z - a) (z - b)}} \ \Big\}.
\end{align*}

\begin{figure}
\begin{center}
\includegraphics{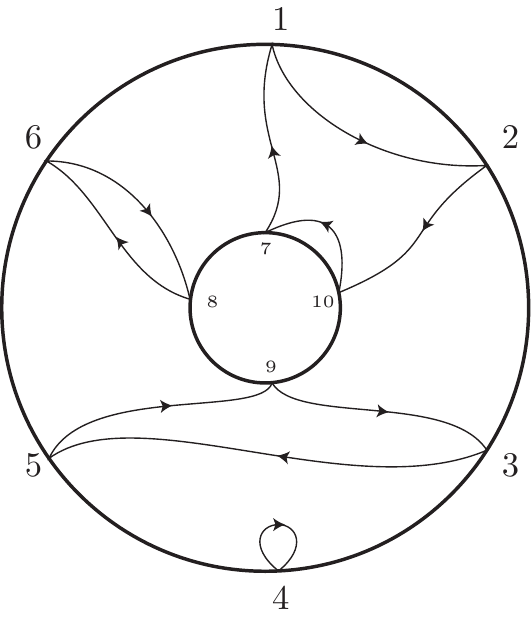}\hfill
\includegraphics{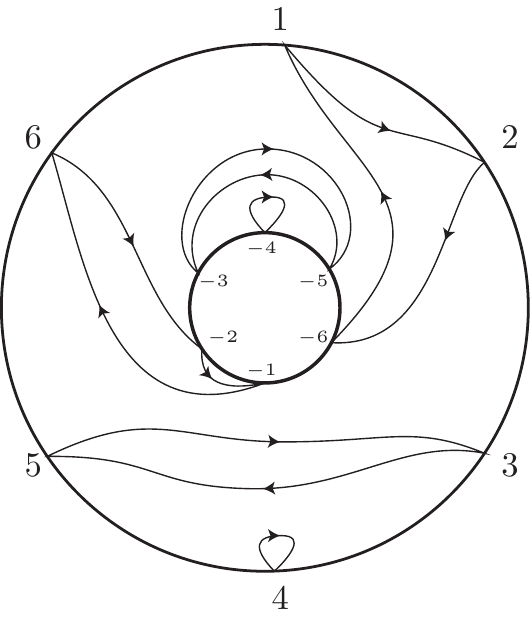}
\end{center}
\caption{\small\label{fig:non-crossing_annular} On the left we have a non-crossing permutation, $\pi$, of the $(6, 4)$-annulus. There are $6$ point on the outer circle and $4$ on the inner circle. The orientation of the inner circle is opposite to that of the outer circle. In cycle form $\pi = (1, 2, 10, 7)(3, 5, 9)(4)\ab(6, 8)$. The cycles of $\pi$ must fit between the two circles and are traversed in a clockwise direction without intersecting another cycle. On the right we have an element of $S_{NC}^\delta(6, -6)$. There are $6$ points on each circle and the circles have the same orientation. Each cycle appears with a mirror image where the order of the points is reversed and the signs are reversed. }
\end{figure}

To complete this introduction let us recall the relation between the fluctuation moments of Wishart matrices and non-crossing annular permutations, as the sum in equation (\ref{eq:m_bar}) is over a subset of the non-crossing annular permutations. Let us recall that $S_{NC}(p, q)$ denotes the set of non-crossing permutations of a $(p, q)$-annulus. These are permutations of $[p+q]$ such that the cycles can be drawn in an annulus, with $p$ points on the outer circle and $q$ points on the inner circle, in such a way that the cycles do not cross, see the left hand figure in Figure \ref{fig:non-crossing_annular}. See \cite[\S5.1]{ms} for a full definition and examples. In \cite{mn} we showed that for a complex Wishart random matrix, $X_N$,  where $c = \lim_N M/N$ we have
\begin{equation}\label{eq:complex_fluctuations}
\lim_M \cov( \Tr(X_N^p), \Tr(X_N^q))
=
\sum_{\mathclap{\pi \in S_{NC}(p, q)}}  c^{\#(\pi)}.
\end{equation}
In \cite[Remark 5.13]{r2} this was extended to the case of real Wishart matrices, provided that the right hand side of (\ref{eq:complex_fluctuations}) is multiplied by $2$. The factor of $2$ is necessary because in the real case we need both possible orientations of the inside circle. Indeed, for the fluctuation moments one  needs the orientation in Figure \ref{fig:non-crossing_annular} where the two circles have opposite orientations, in addition, we need the orientation where the two circles have the same orientation.  When working with the infinitesimal moments we shall \textit{only} need the case where the two circles have the same orientation. We let $[\pm n] = \{\pm 1, \pm, 2, \dots, \pm n\}$, $S_{\pm n}$ be the permutations of $[ \pm n]$ and $\delta \in S_{\pm n}$ be given by $\delta(k) = -k$. If $\sigma \in S_{\pm n}$ we say that $\sigma\delta$ is a pairing mean that all cycles of $\sigma \delta$ have $2$ elements. This, for elements of $S_{NC}^\delta(n, -n)$, is equivalent to requiring that $\delta\sigma\delta^{-1} = \sigma^{-1}$ and that no cycle of $\sigma$ can contain both $k$ and $- k$ for any $k \in [n]$. It also means that for elements of $S_{NC}^\delta(n, -n)$ the cycles always occur in pairs: $c$ and $c'$, where $c'$ is obtained from $c$ by reversing the order and flipping the sign of each entry; see Remark \ref{remark:product_of_pairings}.

The main combinatorial object of this paper will be
\[
S_{NC}^\delta(n, -n) = \{ \sigma \in S_{NC}(n, -n) \mid
\sigma\delta \textrm{ is a pairing }\}.
\]
See the left hand figure in Figure \ref{fig:half_permutations} for an example.

Notice also that the non-crossing permutations and partitions of type B, \cite{bgn} and \cite{no},  all have the property that $\delta \sigma = \sigma \delta$, so we have a different kind of symmetry. 

If we only consider elements of $S_{NC}^\delta(n, -n)$ with blocks of size 2, i.e. pairings then we get the set $NC_2^\delta(n, -n)$ which gave the infinitesimal moments of the GOE, see \cite{m}. Thus the results of this paper, those of F\'eray \cite{f}, and those of \cite{m} show that for orthogonally invariant ensembles, the combinatorics of $S_{NC}^\delta(n, -n)$ are what is needed to describe the corresponding infinitesimal laws.

\subsection*{Outline of the Paper}
After this introduction the sections of this paper are as follows. In Section \ref{sec:notation_preliminaries} we will set out the matrix model we will be using and establish the notation to be used in the rest of the paper. In Section \ref{section:expansion_of_the_trace} we will write the infinitesimal moments in terms of certain pairings. In Section \ref{section:from_pairings_to_permutations} we will show how this sum can be written a sum over $S_{NC}^\delta(n, -n)$. In Section \ref{section:independent_wishart_matrices} we will examine the multi-matrix case and show how this leads to a new kind of independence as found in \cite[Thm. 37]{m}. We are grateful to Alexandru Nica for suggesting this case. In Section  \ref{sec:first_recursion} we establish some preliminaries that will be necessary for Sections \ref{sec:the_subsets}, \ref{sec:the_recursion}, and \ref{sec:the_mp_recursion}. In Section  \ref{sec:the_subsets} we divide $S_{NC}^\delta(n, -n)$ into the three parts, $S_I$, $S_{\mathit{II}}$, and  $S_I$, $S_{\mathit{III}}$, which are necessary for the recursion in Sections  \ref{sec:the_recursion} and \ref{sec:the_mp_recursion}. In two brief concluding sections, \S \ref{sec:infinitesimal_r-transform} and \S \ref{sec:orthogonal_polynomials},  we make some comments on connecting our results with those in the recent paper of Arizmendi, Garza-Vargas, and Perales \cite{agp}.

\section{notation and preliminaries}\label{sec:notation_preliminaries}

In this section we shall describe the matrix model we are using and set up our notation for calculations in the symmetric group.

\begin{definition}\label{def:real_wishart}
$G = (g_{ij})_{ij}$ with $\{ g_{ij}\}_{ij}$ independent
  identically distributed $\cN(0, 1)$ random variables, and
  $1 \leq i \leq N$, $1 \leq j \leq M$.  Let $X_N =
  \frac{1}{N} G G^t$. This is what me mean by a \textit{real Wishart random matrix}. Note that $G$ depends on $M$ and $N$ but we shall suppress this dependency. 

When expanding powers of $X_N$ we shall get products of $G$ and $G^t$, the transpose of $G$. For convenience we shall adopt the following convention: $G^{(1)} = G$ and 
$G^{(-1)} = G^t$. 
\end{definition}

$S_{\pm n}$ is the symmetric group on $[\pm n] = \{\pm 1, \pm 2, \dots, \pm n\}$. We embed $S_n$ into $S_{\pm n}$ be making $\pi \in S_n$ act
trivially on $\{-1, -2, -3, \dots, -n\}$. We shall denote by $\gamma$ the permutation in $S_n$ with the long cycle $(1, 2, \dots, n)$. In this notation $\gamma$ depends on $n$; in case it isn't clear which $n$ is meant we shall write $\gamma_n$. 

$\cP(n)$ is the set of partitions of $[n]$; $\cP_2(2n)$ will denote the pairings of $[2n]$, i.e. a partition with all blocks of size $2$. Every pairing will be thought of as a permutation where each block of size $2$ becomes a transposition. Every permutation can be thought of as the partition whose blocks are the cycles of the permutation. $\#(\pi)$ is the number of cycles of the permutation $\pi$,
  for any $\pi$ and $\sigma$, $\#(\pi \sigma) = \#(\sigma
  \pi)$.

For permutations $\pi$ and $\sigma$ in $S_n$ such that the subgroup $\langle \pi , \sigma \rangle$ acts transitively on $[n]$ there is an integer $ g > 0$ (the genus of a certain surface) such that
\begin{equation}\label{eq:geodesic_equation}
\#(\pi) + \#(\pi^{-1}\sigma) + \#(\sigma) = n + 2(1 - g).
\end{equation}
When $ g = 0$ we say that $\pi$ is non-crossing relative to $\sigma$. When $\sigma = \gamma$, we say that $\pi$ is non-crossing and the set of such permutations is denoted $NC(n)$. See \cite[Lect.~9]{ns} for a full discussion. We shall also be interested in the case when $\sigma$ has two cycles. If $n = p + q$ and $\sigma = (1, 2, \dots, p)(p + 1, \dots, p+ q)$ and $\pi$ is non-crossing relative to $\sigma$ we say that $\pi$ is a non-crossing permutation of the $(p, q)$-annulus and denote by $S_{NC}(p, q)$ the set of such permutations. See \cite{mn} or \cite[Ch.~5]{ms} for further discussion. An example when $p = 6$ and $q = 4$ is shown on the left hand side of Figure \ref{fig:non-crossing_annular}.

The pairing $\delta \in S_{\pm n}$ given by $\delta(k) = -k$ will be central to all of our constructions. When $\sigma = \gamma \delta \gamma^{-1} \delta$ and $\pi \in S_{\pm n}$ is non-crossing relative to $\sigma$ we say that $\pi$ is a non-crossing annular permutation with the reversed orientation and denote by $S_{NC}(n, -n)$ the set of such permutations. The $-n$ is to indicate that the orientation has been reversed on the inner circle. Both the sets $S_{NC}(p, q)$ and $S_{NC}(p, -q)$ figured in the work of Redelmeier on real second order freeness \cite{r2}. 

The main focus of this paper will be on the subset $S_{NC}^\delta(n, -n)$ of $S_{NC}(n,\ab -n )$ consisting of those permutations $\pi$ that satisfy the symmetry condition that $\pi \delta$ is a pairing. Examples when $n = 6$ can be found on the right hand side of Figure \ref{fig:non-crossing_annular}, the left hand side of Figure \ref{fig10:intermediate_pinch}, the top left of Figure 
\ref{fig2:the_first_type}, the top left of Figure \ref{fig3:the_second_type} and the left hand side of Figure \ref{fig3:the_third_type}. 

If $\pi \in S_{NC}^\delta(n, -n)$ and all cycles have length $2$ then $\pi$ will be a non-crossing pairing and the set of such pairings is denoted $NC_2^\delta(n, -n)$. In \cite{m} it was shown that these describe the infinitesimal moments of the Gaussian orthogonal ensemble. 

We will also need the pairing $\omega \in NC_2(2n)$ given by $\omega = (1, 2) (3, 4) \cdots\ab (2n -1 , 2n)$. Again both $\delta$ and $\omega$ depend on $n$. One uses $\omega$ to describe the bijection: $NC_2(2n) \ni \pi \mapsto \omega\pi|_E \in NC(E) \cong NC(n)$ where $E = \{2, 4, 6, \dots, 2n\}$. See Remark \ref{remark:pairings_to_permutations}. The annular version of this is illustrated in Figure \ref{fig:permutation_to_permutations}.  

We set $\bZ_2 = \{-1, 1\}$, let $\epsilon \in \bZ_2^{2n}$ be given
  by $\epsilon = (1, -1, \dots, 1, -1)$. We shall also
  regard $\epsilon$ as a permutation in $S_{\pm 2n}$ as
  follows. 
  
\begin{enumerate}
\item
If $k \in [\pm 2n]$ we let $\epsilon(k) =
  \epsilon_{|k|} k$. As permutations $\epsilon$ and $\delta$
  commute.
\end{enumerate}
For notational convenience we shall write:
\begin{enumerate}\setcounter{enumi}{1}
\item
$\tilde\omega = \omega \delta \omega \delta = (1,2) (-1, -2) \cdots (2n-1, 2n)(-(2n-1), -2n) = \epsilon \omega \delta \omega \epsilon$,

\smallskip\item
$\tg = \gamma \delta \gamma^{-1} \delta$.
\end{enumerate}
We are using the convention that a permutation $\pi \in S_n$ is considered a permutation in $S_{\pm n}$ where $\pi$ acts trivially on negative numbers. Note that 
\begin{enumerate} \setcounter{enumi}{3} 
\item
$\epsilon \gamma \delta \gamma^{-1}\epsilon =
\gamma \omega \gamma^{-1} \delta \omega \delta$. 
\end{enumerate}
Given an $n$-tuple $(j_1, j_2, \dots, j_n)$ with $j_k \in [N]$ for some integer $N$ and $1 \leq k \leq n$, we consider $j$ to be a function from $[n]$ to $[N]$ and its kernel, $\ker(j)$ to be the partition on $[n]$ such that $r \sim_{\ker(j)} s$ if and only if $j_r = j_s$. We use the usual ordering on partitions, namely $\pi \leq \sigma$ means every block of $\pi$ is contained in some block of $\sigma$. For example 
\begin{enumerate} \setcounter{enumi}{4} 
\item
$\ker(j) \geq \gamma \delta \gamma^{-1}$ $\Leftrightarrow $ $j_{-1} =
j_2$, $j_{-2} = j_3$, \dots, $j_{-2n} = j_1$.
\end{enumerate}
When $\pi$ and $\sigma$ are permutations and we write $\pi \vee \sigma$ we considering them to be partition and $\pi \vee \sigma$ is the smallest partition in $\cP(n)$ larger than or equal to $\pi$ and $\sigma$.

\section{Expansion of the trace of a power of $X_N$}
\label{section:expansion_of_the_trace}

In this section we shall write $\E(\tr(X_N^n))$ as a sum over pairings of $[2n]$. This is similar to the calculation done in \cite{glm} and \cite{r1}, but in order to extract the $1/N$ term we repeat it here as we need to have it expressed in our notation. 

\begin{remark}
Given a $4n$-tuple $j = (j_{\pm1}, \dots, j_{\pm 2n})$ where $1
\leq j_1, \dots , j_{2n}\leq N$ and $1 \leq j_{-1}, \dots,
j_{-2n} \leq M$, let $i = j \circ \epsilon$ where $\epsilon$ is as in $(a)$ above. Then
$(G^{(\epsilon_k)})_{j_kj_{-k}} = g_{i_ki_{-k}}$. Also  if $\ker(j) \geq \gamma \delta \gamma^{-1}$, then $\ker(i) \geq \epsilon \gamma
\delta \gamma^{-1} \epsilon = \gamma \omega \gamma^{-1}\,
\delta \omega \delta$. In the equations below we use the
convention that

\centering $\ds\sum_{i_{\pm 1} ,\dots, i_{\pm 2n}}^{(N,M)} =
\sum_{i_1, \dots, i_{2n} =1}^N \sum_{i_{-1}, \dots,
  i_{-2n}=1}^M$.

\end{remark}

\begin{eqnarray*}\lefteqn{
\E(\tr(X_N^n))} \\ 
& = & N^{-(n + 1)}
  \E(\Tr(G^{(\epsilon_1)} \cdots G^{(\epsilon_{2n})})) 
\\ & 
\stackrel{(e)}{=} &
  N^{-(n + 1)}  \mathop{\sum_{j_{\pm 1}, \dots,
      j_{\pm 2n} = 1}^{(N,M)}}_{\ker(j) \geq
  \gamma\delta\gamma^{-1}} \E((G^{(\epsilon_1)})_{j_1j_{-1}}
    \cdots (G^{(\epsilon_{2n})})_{j_{2n}j_{-2n}}) \\ 
& = & N^{-(n +
      1)} \mathop{\sum_{i_{\pm 1}, \dots,
        i_{\pm 2n} = 1}^{(N,M)}}_{\ker(i) \geq
    \epsilon\gamma\delta\gamma^{-1}\epsilon}
      \E(g_{i_1i_{-1}} \cdots g_{i_{2n}i_{-2n}} ).
\end{eqnarray*}

Now $\E(g_{i_1i_{-1}} \cdots g_{i_{2n}i_{-2n}} ) = |\{\pi \in
\cP_2(2n) \mid i_r = i_s$ and $i_{-r} = i_{-s}$ whenever $(r,
s) \in \pi\}) = |\{\pi \mid \ker(i) \geq \pi \delta \pi \delta\}|$. Thus

\begin{eqnarray*}\lefteqn{%
\mathop{\sum_{i_{\pm 1}, \dots, i_{\pm 2n}}}_%
{\mathclap{\ker(i)\geq \gamma \omega \gamma^{-1} \delta \omega \delta}}
  \E(g_{i_1i_{-1}} \cdots g_{i_{2n}i_{-2n}} ) 
= 
\sum_{\mathclap{\pi \in
    \cP_2(2n)}} |\{ i \mid \ker(i) \geq \gamma \omega \gamma^{-1} \delta \omega \delta \vee \pi \delta \pi \delta \}|  }\\
&=&
\sum_{\pi \in \cP_2(2n)} M^{\#(\delta \omega \delta \vee \delta \pi \delta)}
N^{\#(\gamma \omega \gamma^{-1} \vee \pi)} 
=
\sum_{\pi \in \cP_2(2n)} M^{\#(\omega  \vee  \pi )}
N^{\#(\gamma \omega \gamma^{-1} \vee \pi)} \\
&=&
\sum_{\pi \in \cP_2(2n)} \Big(\frac{M}{N}\Big)%
^{\#(\omega  \vee  \pi )}
N^{\#(\gamma \omega \gamma^{-1} \vee \pi) + \#(\omega  \vee  \pi )}.
\end{eqnarray*} 

We now have proved the following lemma.
\begin{lemma}\label{lemma:1}
\begin{equation}\label{equation:1}
\E(\tr(X_N^n)) =  
\sum_{\mathclap{\pi \in \cP_2(2n)}} \ \ \Big(\frac{M}{N}\Big)%
^{\#(\omega  \vee  \pi )}
N^{\#(\gamma \omega \gamma^{-1} \vee \pi) + \#(\omega  \vee  \pi ) -(n+1)}
\end{equation}
\end{lemma}

\begin{remark}\label{remark:product_of_pairings}
We have to decide for a pairing $\pi$ what the maximum value
of $\#(\gamma \omega \gamma^{-1} \vee \pi) + \#(\omega \vee
\pi ) -(n+1)$ is. Recall that if $p$ and $q$ are pairings
and $p \vee q$ denotes the join as partitions then $2(p \vee
q) = \#(pq)$. Moreover we can write the cycle decomposition
of $pq$ as $pq = c_1c_1' \cdots c_k c_k'$ where $c_l' = q
c_l^{-1} q$ and the blocks of $p \vee q$ are $\{c_1 \cup c_1', \dots, c_k \cup c'_k\}$, see \cite[Lemma 2]{mp}. 
\end{remark}

\begin{remark}\label{remark:sum_over_pairings}
In the proof of Lemma \ref{lemma:1} we have established the following expansion, writing $\ols\pi = \gamma^{-1} \pi \gamma$,  
\[
\E(\Tr(X_N^n)) = N^{-n} \kern-0.5em
\sum_{\pi \in \cP_2(2n)} M^{\#(\omega \vee \pi)}
N^{\#(\omega \vee \ols\pi)}. 
\]
For a complex Wishart matrix, $Y_N$, we have
\[
\E(\Tr(Y_N^n)) = N^{-n} 
\sum_{\sigma \in S_n} M^{\#(\sigma)}
N^{\#(\sigma^{-1}\gamma)}. 
\]
This suggests that $\ols\pi$ plays the role of the Kreweras complement of $\pi$. Indeed, the leading term of equation (\ref{equation:1}) is given by those $\pi$'s for which $\#(\omega \vee \pi) + \#(\omega \vee \ols\pi) = n+1$; these are exactly the non-crossing pairings. The subleading term of equation (\ref{equation:1}) must be exactly those $\pi$'s for which $\#(\omega \vee \pi) + \#(\omega \vee \ols\pi) = n$. These $\pi$'s are not planar on the circle, but will be planar when drawn in an annulus. See Figure \ref{fig:permutation_to_permutations} for an example, and Notation \ref{notation:tilde_gamma} for additional comments on the Kreweras complement.
\end{remark}

\setbox1=\hbox{%
\begin{tikzpicture}[anchor=base, baseline]
\node[above] at (0.5,.45)  {$1$};
\node[above] at (1.0,.45)  {$2$};
\node[above] at (1.5,.45)  {$3$};
\node[above] at (2.0,.45)  {$4$};
\node[above] at (2.5,.45)  {$5$};
\node[above] at (3.0,.45)  {$6$};
\node[above] at (3.5,.45)  {$7$};
\node[above] at (4.0,.45)  {$8$};
\node[above] at (4.5,.45)  {$9$};
\node[above] at (5.0,.45)  {$10$};
\node[above] at (5.5,.45)  {$11$};
\node[above] at (6.0,.45)  {$12$};
\node[above] at (6.5,.45)  {$13$};
\node[above] at (7.0,.45)  {$14$};
\node[above] at (7.5,.45)  {$15$};
\node[above] at (8.0,.45)  {$16$};
\node[above] at (8.5,.45)  {$17$};
\node[above] at (9.0,.45)  {$18$};
\node[above] at (9.5,.45)  {$19$};
\node[above] at (10.0,.45) {$20$};
\node[above] at (10.5,.45)  {$21$};
\node[above] at (11.0,.45) {$22$};
\draw [thick]       (0.5,1.0) -- (0.5, 1.3) -- (1.0, 1.3) -- (1.0,1.0);
\draw [thick]       (1.5,1.0) -- (1.5, 1.3) -- (2.0, 1.3) -- (2.0,1.0);
\draw [thick]       (2.5,1.0) -- (2.5, 1.3) -- (3.0, 1.3) -- (3.0,1.0);
\draw [thick]       (3.5,1.0) -- (3.5, 1.3) -- (4.0, 1.3) -- (4.0,1.0);
\draw [thick]       (4.5,1.0) -- (4.5, 1.3) -- (5.0, 1.3) -- (5.0,1.0);
\draw [thick]       (5.5,1.0) -- (5.5, 1.3) -- (6.0, 1.3) -- (6.0,1.0);
\draw [thick]       (6.5,1.0) -- (6.5, 1.3) -- (7.0, 1.3) -- (7.0,1.0);
\draw [thick]       (7.5,1.0) -- (7.5, 1.3) -- (8.0, 1.3) -- (8.0,1.0);
\draw [thick]       (8.5,1.0) -- (8.5, 1.3) -- (9.0, 1.3) -- (9.0,1.0);
\draw [thick]       (9.5,1.0) -- (9.5, 1.3) -- (10.0, 1.3) -- (10.0,1.0);
\draw [thick]       (10.5,1.0) -- (10.5, 1.3) -- (11.0, 1.3) -- (11.0,1.0);
\draw [ultra thick] (0.5, 0.5) -- (0.5, -0.2) -- (6.5, -0.2) -- (6.5, 0.5);
\draw [thick]       (1.0, 0.5) -- (1.0, 0.2) -- (1.5, 0.2) -- (1.5, 0.5);
\draw [thick]       (2.0, 0.5) -- (2.0, 0.2) -- (2.5 ,0.2) -- (2.5, 0.5);

\draw [ultra thick] (3.0, 0.5) -- (3.0, -0.4) -- (7.0, -0.4) -- (7.0, 0.5);
\draw [ultra thick] (3.5, 0.5) -- (3.5, -0.6) -- (7.5 , -0.6) -- (7.5, 0.5);
\draw [ultra thick] (4.0, 0.5) -- (4.0, -0.8) -- (8.0, -0.8) -- (8.0, 0.5);
\draw [thick]       (4.5 , 0.5) -- (4.5, -0.0) -- (6.0 , -0.0) -- (6.0, 0.5);
\draw [thick]       (5.0, 0.5) -- (5.0, 0.2) -- (5.5, 0.2) -- (5.5, 0.5);
\draw [thick]       (9.0, 0.5) -- (9.0, 0.2) -- (9.5, 0.2) -- (9.5, 0.5);
\draw [thick]       (8.5, 0.5) -- (8.5, 0.0) -- (10.0, 0.0) -- (10.0, 0.5);
\draw [thick]       (10.5, 0.5) -- (10.5, 0.2) -- (11, 0.2) -- (11, 0.5);
\end{tikzpicture}}

\begin{figure}
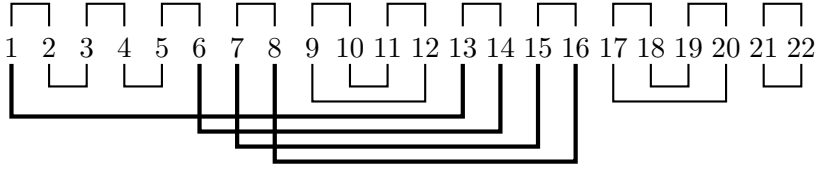

\begin{center}\leavevmode\box1\end{center}
\caption{\label{fig:pi_and_omega}\small We let $I_1 = \{2, 3, 4, 5\}$, $I_2 = \{9, 10, 11, 12\}$, $I_3 = \{17, 18, 19, 20, 21,  \ab 22\}$ and then $J = \{1, 6, 7, 8, 13, 14, 15, 16\}$ is the complement of $I_1 \cup I_2 \cup I_3$. To make a pairing $\pi$ such that $\#(\omega \vee \pi) + \#(\ols\omega \vee \pi) = 22/2$ we take the  pairing of $J$: $(1, 13)$, $(6, 14)$, $(7, 15)$, and $(8, 16)$ of $J$, shown by ht 4 thick lines above. Then we choose non-crossing pairings of $I_1$, $I_2$, and $I_3$, shown by the thin lines. This produces $\pi = \{(\bm 1, \bm {13}),  (2, 3),  (4, 5),  (\bm 6, \bm{14}), (\bm 7, \bm{15}), \ab (\bm 8, \bm{16}),   (9, 12),(10, 11), (17, 20), (18, 19), (21, 22)\}$. On the top row we have $\omega = \{(1,2),  (3, 4), \ab \dots, (21, 22)\}$.   We have $\#(\omega \vee \pi)\ab  = 5$.} 
\end{figure}

\bigskip\bigskip

\setbox2=\hbox{%
\begin{tikzpicture}[anchor=base, baseline]
\node[above] at (0.5,.45)  {$1$};
\node[above] at (1.0,.45)  {$2$};
\node[above] at (1.5,.45)  {$3$};
\node[above] at (2.0,.45)  {$4$};
\node[above] at (2.5,.45)  {$5$};
\node[above] at (3.0,.45)  {$6$};
\node[above] at (3.5,.45)  {$7$};
\node[above] at (4.0,.45)  {$8$};
\node[above] at (4.5,.45)  {$9$};
\node[above] at (5.0,.45)  {$10$};
\node[above] at (5.5,.45)  {$11$};
\node[above] at (6.0,.45)  {$12$};
\node[above] at (6.5,.45)  {$13$};
\node[above] at (7.0,.45)  {$14$};
\node[above] at (7.5,.45)  {$15$};
\node[above] at (8.0,.45)  {$16$};
\node[above] at (8.5,.45)  {$17$};
\node[above] at (9.0,.45)  {$18$};
\node[above] at (9.5,.45)  {$19$};
\node[above] at (10.0,.45) {$20$};
\node[above] at (10.5,.45)  {$21$};
\node[above] at (11.0,.45) {$22$};
\draw [thick]       (0.5,1.0) -- (0.5, 1.5) -- (11.0, 1.5) -- (11.0,1.0);
\draw [thick]       (1.0,1.0) -- (1.0, 1.3) -- (1.5, 1.3) -- (1.5,1.0);
\draw [thick]       (2.0,1.0) -- (2.0, 1.3) -- (2.5, 1.3) -- (2.5,1.0);
\draw [thick]       (3.0,1.0) -- (3.0, 1.3) -- (3.5, 1.3) -- (3.5,1.0);
\draw [thick]       (4.0,1.0) -- (4.0, 1.3) -- (4.5, 1.3) -- (4.5,1.0);
\draw [thick]       (5.0,1.0) -- (5.0, 1.3) -- (5.5, 1.3) -- (5.5,1.0);
\draw [thick]       (6.0,1.0) -- (6.0, 1.3) -- (6.5, 1.3) -- (6.5,1.0);
\draw [thick]       (7.0,1.0) -- (7.0, 1.3) -- (7.5, 1.3) -- (7.5,1.0);
\draw [thick]       (8.0,1.0) -- (8.0, 1.3) -- (8.5, 1.3) -- (8.5,1.0);
\draw [thick]       (9.0,1.0) -- (9.0, 1.3) -- (9.5, 1.3) -- (9.5,1.0);
\draw [thick]       (10.0,1.0) -- (10., 1.3) -- (10.5, 1.3) -- (10.5,1.0);

\draw [ultra thick] (0.5, 0.5) -- (0.5, -0.2) -- (6.5, -0.2) -- (6.5, 0.5);
\draw [thick]       (1.0, 0.5) -- (1.0, 0.2) -- (1.5, 0.2) -- (1.5, 0.5);
\draw [thick]       (2.0, 0.5) -- (2.0, 0.2) -- (2.5 ,0.2) -- (2.5, 0.5);

\draw [ultra thick] (3.0, 0.5) -- (3.0, -0.4) -- (7.0, -0.4) -- (7.0, 0.5);
\draw [ultra thick] (3.5, 0.5) -- (3.5, -0.6) -- (7.5 , -0.6) -- (7.5, 0.5);
\draw [ultra thick] (4.0, 0.5) -- (4.0, -0.8) -- (8.0, -0.8) -- (8.0, 0.5);
\draw [thick]       (4.5 , 0.5) -- (4.5, -0.0) -- (6.0 , -0.0) -- (6.0, 0.5);
\draw [thick]       (5.0, 0.5) -- (5.0, 0.2) -- (5.5, 0.2) -- (5.5, 0.5);
\draw [thick]       (9.0, 0.5) -- (9.0, 0.2) -- (9.5, 0.2) -- (9.5, 0.5);
\draw [thick]       (8.5, 0.5) -- (8.5, 0.0) -- (10.0, 0.0) -- (10.0, 0.5);
\draw [thick]       (10.5, 0.5) -- (10.5, 0.2) -- (11, 0.2) -- (11, 0.5);
\end{tikzpicture}}

\begin{figure}
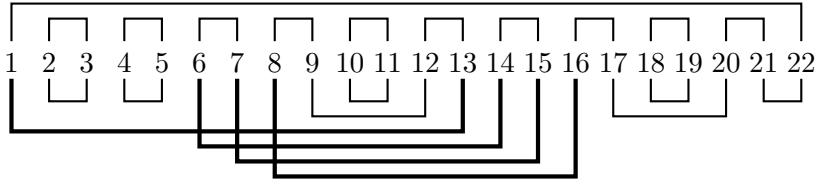

\begin{center}\leavevmode\box2\end{center}
\caption{\label{fig:pi_and_bar_omega}\small 
On the top row we have $\ols\omega = \gamma^{-1}\omega \gamma$. We have $\#(\ols\omega \vee \pi)\ab  = 6$. Thus $\#(\omega \vee \pi) + \#(\ols\omega \vee \pi)\ab  = 11$. So this $\pi$ contributes to the $N^{-1}$ term, i.e the infinitesimal  term. See Remark \ref{remark:construction_of_pi}}. 
\end{figure}

\begin{remark}\label{remark:construction_of_pi}
There is a simple way to generate all possible examples of $\pi \in \cP_2(2n)$ such that $\#(\omega \vee \pi) + \#(\omega \vee \ols\pi) = n$. Given $n$, choose disjoint intervals, each of even length, $I_1, \dots, I_k \subseteq [2n]$ such that the cardinality of $J$, the complement of $\cup_{i=1}^k I_k$, is divisible by $4$. Let $J = \{ j_1, \dots, j_m\}$ with $m = 2p$,  and consider the pairs $(j_1, j_{p+1}), \dots, (j_l, j_{p+l}), \dots, (j_p, j_{2p})$. Choose a non-crossing pairing for each interval $I_1, \dots, I_k$. Together these will give a $\pi$ satisfying $\#(\omega \vee \pi) + \#(\omega \vee \ols\pi) = n$. This is illustrated in Figures \ref{fig:pi_and_omega} and \ref{fig:pi_and_bar_omega}. 
\end{remark}

\begin{lemma}
\begin{enumerate}
\item

For $\pi \in \cP_2(2n)$ we have $\#(\gamma \omega
\gamma^{-1} \vee \pi) + \#(\omega \vee \pi ) -(n+1) \leq 0$,
with equality only if for all $(r, s) \in \pi$ we have
$\epsilon_r = - \epsilon_s$ and $\pi$ is a non-crossing
pairing of $[2n]$.

\item
If there is $(r, s) \in \pi$ with $\epsilon_r = \epsilon_s$
then $\#(\gamma \omega \gamma^{-1} \vee \pi) + \#(\omega
\vee \pi ) -(n+1) \leq -1$, with equality only if
$\epsilon\pi\delta\pi \epsilon$ is a non-crossing pairing of
a $(2n, -2n)$-annulus.
\end{enumerate}
\end{lemma}

\begin{proof}
In the equations below we use `$\cdot$' to separate two
expressions being multiplied, just to make the reading
easier.
\begin{eqnarray*}\lefteqn{%
2( \#(\gamma \omega \gamma^{-1} \vee \pi) + \#(\omega  \vee  \pi )) =
\#(\delta \gamma \omega \gamma^{-1} \delta \cdot \delta \pi \delta) +
\#(\omega \pi)} \\
& = &
\#( \omega \pi \delta \gamma \omega \gamma^{-1} \delta \cdot \delta \pi \delta) 
 = 
\#(\omega \delta \gamma \omega \gamma^{-1} \delta \, \pi \delta \pi \delta ) \\
& = &
\#( \epsilon \gamma \delta \gamma^{-1} \epsilon \cdot \pi \delta \pi \delta) 
=
\#( \gamma \delta \gamma^{-1} \delta \cdot \epsilon \pi \delta \pi \epsilon).
\end{eqnarray*}
In particular $\#(\gamma \omega \gamma^{-1} \vee \pi) +
\#(\omega \vee \pi) = \#(\epsilon \gamma \delta \gamma^{-1}
\epsilon \vee \pi \delta \pi \delta)$.

Now $\gamma \delta
\gamma^{-1} \delta$ has 2 cycles and $\epsilon \pi \delta
\pi \epsilon$ is a pairing. Thus $\#(\gamma \delta
\gamma^{-1} \delta) = 2$ and $\#(\epsilon \pi \delta \pi
\epsilon) = 2n$.

Next we consider two cases. In the first case for all $(r,
s) \in \pi$ we have $\epsilon_r = -\epsilon_s$. Then
$\epsilon \pi \delta \pi \epsilon = \pi \delta \pi
\delta$. In this case
\[
\#(\gamma \delta \gamma^{-1} \delta \epsilon \pi \delta \pi \epsilon)
=
\#(\gamma \delta \gamma^{-1} \delta \pi \delta \pi \delta)
=
\#(\gamma\pi \delta \gamma^{-1}\pi \delta)
 = 2 \#(\gamma \pi).
\]
So by Equation \ref{eq:geodesic_equation} we have for some integer $g \geq 0$
\[
\#(\pi) + \#(\gamma\pi) + \#(\gamma) = 2n + 2(1 - g).
\]
So
\[
\#(\gamma\pi) = n + 1 - 2g.
\]
Thus
\begin{eqnarray*}
\#( \epsilon \gamma \delta \gamma^{-1} \epsilon \vee \pi
\delta \pi \delta) \ab -(n + 1) = \#(\gamma\pi) - (n +
1) = -2 g \leq 0
\end{eqnarray*}
So in the first case we get $\#(\gamma \omega \gamma^{-1}
\vee \pi) + \#(\omega \vee \pi ) -(n+1) = 0$ only when $\pi$
is non-crossing and $-2$ or less when $\pi$ has a
crossing. This proves $(a)$.

In the second case there is some $(r, s) \in \pi$ such that
$\epsilon_r = \epsilon_s$. In this case $\langle \gamma
\delta\gamma^{-1} \delta, \epsilon \pi \delta \pi \epsilon
\rangle$ acts transitively on $[\pm 2n]$. So again by Equation
\ref{eq:geodesic_equation}  we have for some integer $g' \geq 0$
\[
\#( \gamma \delta \gamma^{-1} \delta \epsilon \pi \delta \pi
\epsilon ) + \#(\epsilon \pi \delta \pi \epsilon ) +
\#(\gamma \delta \gamma^{-1} \delta) = 4n + 2(1 - g')
\]
Thus
\[
\#( \epsilon \gamma \delta \gamma^{-1} \epsilon \vee \pi
\delta \pi \delta) \ab -(n + 1) = -1 - 2 g' \leq -1.
\]
Moreover $g' = 0$ only if $\epsilon\pi\delta\pi \epsilon$ is
a non-crossing pairing of a $(2n, -2n)$-annulus. This proves
$(b)$.
\end{proof}

\begin{remark}\label{remark:pairings_to_permutations}
We have identified the leading term as all the non-crossing
pairings $\pi$ where $\epsilon_r = - \epsilon_s$ for all
$(r, s) \in \pi$. But the second condition is automatic for
a non-crossing pairing. Thus the $N^0$ term is exactly the
part of the sum where $\pi \in NC_2(2n)$. Recall that there
is a bijection from $NC_2(2n) \ni \pi \mapsto \sigma_\pi \in
NC(n)$ and under this bijection $\#(\omega \vee \pi) =
\#(\sigma_\pi)$. The map is given explicitly as follows.

For the purposes of this remark let $E = \{2, 4, 6, \dots,
2n\}$ and $O = \{1,3,5,\ab \dots, 2n-1\}$; this use is at
variance with what we use in Notation
\ref{notation:E_and_O}, but the discussion will be
illustrative. Note that $\gamma^2$ leaves $O$ and $E$
invariant, we denote by $\gamma_E$ the restriction of
$\gamma^2$ to $E$.

Both $\pi$ and $\omega$ map $E$ onto $O$ and vice
versa. Thus $\omega\pi$ leaves both $E$ and $O$
invariant. Let $\sigma_\pi$ be the restriction of
$\omega\pi$ to $E$. Since both $\pi$ and $\omega$ are
pairings we have $\#(\omega \vee \pi) =
\#(\sigma_\pi)$. Moreover we have the following $\omega
\gamma |_E = \gamma^2|_E$, so $\pi \omega \gamma^2|_E = \pi
\omega \omega \gamma|_E = \pi \gamma|_E$.  Also, as maps
from $O$ to $E$, we have $\gamma\omega\gamma^{-1} |_O =
\gamma^{-1}|_O$. Thus $\gamma\omega \gamma^{-1} \pi |_E =
\gamma \omega \gamma^{-1}|_O \pi|_E = \gamma^{-1}|_O \pi |_E
= \gamma^{-1}\pi|_E$. Hence $\#(\sigma_\pi^{-1}\gamma_E) =
\#(\pi \omega \gamma^2|_E) = \#(\gamma^{-1} \pi|_E) =
\#(\gamma \omega \gamma^{-1} \pi |_E) = \#(\gamma\omega
\gamma^{-1} \vee \pi)$.

Since $\#(\pi \vee \omega) + \#(\pi \vee \gamma \omega
\gamma^{-1}) = n + 1$ we have that $\#(\sigma_\pi)
+ \#(\sigma_\pi^{-1} \gamma_E) = n + 1$ and thus,
$\sigma_\pi$ is non-crossing, and $\#(\omega \vee \pi) =
\#(\sigma_\pi)$.

Conversely given $\sigma \in NC(E)$, let $\pi_\sigma =
\sigma^{-1} \omega \sigma$ be a pairing of $[2n]$. Here we
using the convention that $\sigma$ acts trivially on
$O$. Then $\omega \pi_\sigma|_E = \omega\sigma^{-1}\omega
\sigma|_E = \sigma$. Also we have $\#(\omega \vee
\pi_\sigma) = \#(\sigma)$ and $\#(\gamma \omega \gamma^{-1}
\vee \pi_\sigma) = \#(\sigma^{-1} \gamma_E)$. Thus
$\#(\omega \vee \pi_\sigma) + \#(\gamma \omega \gamma^{-1}
\vee \pi_\sigma) = n+1$. So $\pi_\sigma$ is
non-crossing. This gives the bijection between the
non-crossing pairings of $[2n]$ and the non-crossing
partitions of $E$. In Proposition \ref{prop:annular_bijection} we
shall extend this construction to the annular case.
\end{remark}

\begin{corollary}
The only pairings $\pi$ that can contribute to the
coefficient of $N^{-1}$ in $(\ref{equation:1})$ are those for
which there is at least one $(r, s) \in \pi$ such that
$\epsilon_r = \epsilon_s$.
\end{corollary}

We have already seen that for such a pair we have the
subgroup $\langle \gamma \delta\gamma^{-1} \delta,\ab \epsilon
\pi \delta \pi \epsilon \rangle$ acts transitively on $[\pm
  2n]$ and so we may apply Equation \ref{eq:geodesic_equation}
to conclude
\[
\#( \gamma \delta \gamma^{-1} \delta \epsilon \pi \delta \pi
\epsilon ) + \#(\epsilon \pi \delta \pi \epsilon ) +
\#(\gamma \delta \gamma^{-1} \delta) = 4n + 2(1 - g')
\]
for some $g' \geq 0$. Hence 
\[
\#( \epsilon \gamma \delta \gamma^{-1} \epsilon \vee \pi
\delta \pi \delta) \ab -(n + 1) = -1 - g'
\]
So we seek all those pairings $\pi$ such that $g' =
0$. This means that $\epsilon \pi \delta \pi \epsilon$ is
non-crossing with respect to $\gamma \delta \gamma^{-1}
\delta = (1, 2,\ab \dots,\ab  2n)\ab (-2n, -(2n-1), \dots,\ab -2, -1)$. Thus
$\epsilon \pi \delta \pi \epsilon$ must be a non-crossing
annular pairing (see \cite[Thm. 6.1]{mn} and \cite[\S 4]{m}) such that
\begin{enumerate}[$(i)$]

\item
$\epsilon \pi \delta \pi \epsilon$ connects the two circles

\item
$\epsilon \pi \delta \pi \epsilon$ commutes with $\delta$

\item
the pairs of $\epsilon \pi \delta \pi \epsilon$ come in
pairs. For $(r, s) \in \pi$ with $\epsilon_r =
-\epsilon_s$, $\{(r, s),(-r, -s)\}$ are the corresponding
pairs of $\epsilon \pi \delta \pi \epsilon$. This produces a
pair of strings that do \textit{not} connect the two
cycles: $\gamma$ and $\delta\gamma^{-1}\delta$.

\item
For $(r, s) \in \pi$ with $\epsilon_r = \epsilon_s$, $\{(r, -s),(-r, s)\}$ are the
corresponding pairs of $\epsilon \pi \delta \pi
\epsilon$. This produces a pair of strings that \textit{do}
connect the two cycles: $\gamma$ and
$\delta\gamma^{-1}\delta$, i.e. \textit{through strings}.

\end{enumerate}

\begin{notation}
We denote by $NC_2^\delta(2n, -2n)$ the set of non-crossing annular pairings that satisfy (\textit{i}), (\textit{ii}), (\textit{iii}), and (\textit{iv}) above. 
\end{notation}

Combining the previous results we have with $\epsilon = (1 -1, 1, -1, \dots, 1, -1)\ab \in \bZ_2^{2n}$ the following theorem.
\begin{theorem}\label{thm:first_main_claim}
\begin{eqnarray*}\lefteqn{
m'_n = \lim_{N \rightarrow \infty}
N\big( \E(\tr(X^n)) - \sum_{\mathclap{\pi \in NC(n)}} c^{\#(\pi)} \big)}\\ 
&= & 
\sum_{\mathclap{\pi \in NC(n)}} c' \#(\pi) c^{\#(\pi)-1}
+ \quad
\mathop{\sum_{\mathclap{\pi \in \cP_2(2n)}}}_{\mathclap{\epsilon\pi\delta\pi\epsilon
\in NC_2^\delta(2n , -2n)}}\ 
c^{\#(\omega \vee \pi)}.
\end{eqnarray*}
\end{theorem}

\section{from pairings to permutations}\label{section:from_pairings_to_permutations}

In this section we shall focus on the second term in the right hand side of Theorem \ref{thm:first_main_claim}. We want to rewrite this as a sum over $S_{NC}^\delta(n, -n)$. This is done in Proposition \ref{prop:annular_bijection}. In Proposition \ref{prop:counting_diagrams} we write this term as a polynomial in $c$ and give an explicit formula for each coefficient.

\begin{lemma}\label{lemma.structure.through.pairings}
Let $\pi \in \cP_2(2n)$ and $\rho = \epsilon \pi \delta \pi \epsilon \in NC_2^\delta(2n, -2n)$. Suppose the through strings of $\rho$ are $(r_1, -s_1), (r_2, -s_2), \dots,\ab (r_k, \ab -s_k)$, with $0 < r_1 < r_2 < \cdots < r_k$. Then
\begin{enumerate}
\item
$k$ is even; let $l = k/2$;

\item
for $1 \leq i \leq l$, $r_i = s_{l+i}$ and $s_i = r_{l+i}$;

\item
for $1 \leq i < k$, $r_i$ and $r_{i+1}$ have opposite parity;

\item
for all $i$, $r_i$ and $s_i$ have the same parity;

\item
$l$ is even.

\end{enumerate}
\smallskip
\end{lemma}

\begin{proof}
Since $(r_i, -s_i) \in \rho$ we have $(r_i, s_i) \in \pi$, and in particular $r_i \not = s_i$. Also since $(r, -s) \in \rho \Leftrightarrow (s, -r) \in \rho$, we have that for each $i$ there is $j$ such that $r_i = s_j$. Let $1 \leq t \leq k$ be such that $s_1 = r_t$. The through strings of $\rho$ must form a spoke diagram; so $s_2 = r_{t+1}$, $s_3 = t_{t+2}$, \dots, $s_i = r_{t+i-1}$ with the indices interpreted modulo $k$. Thus for all $i$, $(r_i, -r_{t+i-1}) \in \rho$ and $(r_{t+i-1}, -r_i) \in \rho$. Hence $r_i = s_{t+i-1} = r_{2(t-1)+i}$.  Hence $2(t-1) \equiv  0$ (mod $k$). We cannot have $t = 1$, otherwise $r_i = s_i$; so we must have $k = 2(t-1)$ and thus $k = 2l $ with $l = t-1$. Thus the through strings are $(r_1, -r_{l+1}), (r_2, -r_{l+2}), \dots, (r_{2l}, -r_l)$. This proves $(a)$ and $(b)$. 

Between $r_i$ and $r_{i+1}$ there are blocks of $\rho$ which do not cross any other pairs of $\rho$, hence there is an even number of points in the gap. Thus $r_i$ and $r_{i+1}$ have opposite parity. This proves $(c)$. 

If $(r, s) \in \pi$ then $(\epsilon(r), -\epsilon(s)) \in \rho$. If  $(\epsilon(r), -\epsilon(s))$ is a through string and $r$ is odd then we need $\epsilon(s) = s$ so $s$ must be odd as well. Likewise if $r$ is even, $s$ must be even. This proves $(d)$. Thus $r_1$ and $r_{l+1}$ will have the same parity so by $(c)$, $l$ must be even.
\end{proof}

A converse to the previous lemma is the following construction of all the pairings $\pi \in \cP_2(2n)$ that satisfy the condition $ \epsilon \pi \delta \pi \epsilon \in NC^\delta_2(n,-n)$. 
First, we choose mutually disjoint intervals of even length $I_1, I_2, \ldots, I_m \subset [2n]$ with the constraint that $4$ divides $\vert [2n] \setminus \cup_{j=1}^m I_j \vert $ and take $\{r_1 < r_2 < \cdots < r_{4t} \} = [2n] \setminus \cup_{j=1}^m I_j$. 
Note that the $r$'s alternate in parity since each $I_j$ has even length. 
Second, we choose non-crossing parings $\pi_{j} \in NC_{2}(I_{j})$ for $j = 1,2,\ldots,m$. 
Finally, taking 
$$ \pi = \pi_1 \pi_2 \cdots \pi_m   (r_1,r_{2t+1}) (r_2,r_{2t+2}) \cdots (r_{2t},r_{4t}),$$ 
we obtain obtain a pairing that satisfies $ \epsilon \pi \delta \pi \epsilon \in NC^\delta_2(n,-n)$.  

\begin{notation}\label{notation:E_and_O}
Let $E = \{2, 4, 6, \dots, 2n\} \cup \{-1, -3, -5, \dots, -(2n-1)\}$ and $O = \{1, 3, 5, \dots, 2n -1\} \cup \{ -2, -4, -6, \dots, -2n\}$. Recall that $\tilde \omega = \omega \delta \omega \delta = \epsilon\omega\delta\omega\epsilon$ and $\tilde\gamma = \gamma \delta \gamma^{-1} \delta$. This extends the notation used in Remark \ref{remark:pairings_to_permutations}. 
\end{notation}

\begin{remark}\label{remark:parity_reversal}
$\tilde\omega(O) = E$, $\tilde\omega(E) = O$. If $\rho \in NC_2^\delta(2n , -2n)$ then $\rho(O) = E$ and $\rho(E) = O$. Thus $\tilde\omega\rho(E) = E$. We also have that $\tg(O) = E$ and $\tg(E) = O$. Moreover $\tg|_O = \tom|_O$ and $\tg^{-1}|_E = \tom|_E$. 
\end{remark}

\begin{lemma}\label{lemma:the_restrictions}
If $\pi \in \cP_2(2n)$ and $\rho = \epsilon\pi\delta\pi\epsilon \in NC_2^\delta(2n -2n)$ then 
\begin{enumerate}

\item
$\#(\omega\pi) = \#(\tilde\omega\rho|_E)$,

\item
$\#(\omega \vee \pi) = 1/2 \#(\tilde\omega\rho|_E)$,

\item
$\#(\tom\rho|_E) + \#(\rho\tom\cdot \tg^2|_E) = 2n$.

\end{enumerate}
\end{lemma}

\begin{proof}
When we decompose $\tom\rho = c_1 c_1' \cdots c_k c_k'$ we have, by Remark \ref{remark:parity_reversal}, each cycle is either in $E$ or in $O$. Moreover for each pair $\{ c_i, c_i'\}$, if one is in $E$ then the other must be in $O$.  Also
\begin{eqnarray*}
\#(\rho\tom) 
& = &
\#(\epsilon \pi \delta \pi \epsilon \cdot \epsilon\omega\delta\omega\epsilon ) \\
& = &
\#(\epsilon \omega \delta \omega \pi \delta \pi \epsilon) 
=
\#(\omega \pi \delta \pi\omega  \delta )
=
2\, \#(\omega \pi). 
\end{eqnarray*}
Putting these together we get $\#(\omega \pi) = 1/2 \#(\rho\tom) = \#(\rho\tom|_E)$. Finally we have $\#(\omega \vee \pi) = 1/2 \, \#(\omega\pi) = 1/2\, \#(\rho\tom|_E)$. Thus $\#(\rho\tom|_E) = \#(\rho\tom|_O) = \frac{1}{2} \#(\rho\tom)$. This proves $(a)$ and $(b)$.

Because we assumed that $\rho \in NC_2^\delta(2n , -2n)$ we have $2n = \#(\rho\tg) = \#(\rho\tg|_E) + \#(\rho\tg|_O)$. Since $\tg|_O = \tom|_O$, we then have $\#(\rho\tg|_O) = \#(\rho\tom|_O) = \#(\rho\tom|_E)$.  Again as $\tg|_O = \tom|_O$, we have $\tg^2|_E = \tom \tg |_E$, thus $\rho \tom\cdot \tg^2|_E = \rho\tom \cdot \tom \tg|_E = \rho\tg|_E$. Hence
\begin{eqnarray*}
2 n 
& = &
\#(\rho\tg|_O) + \#(\rho \tg|_E)\\
& = &
\#(\rho\tom|_E) + \#(\rho\tom \cdot \tg^2|_E).
\end{eqnarray*}
This proves $(c)$. 
\end{proof}

\begin{figure}
\includegraphics{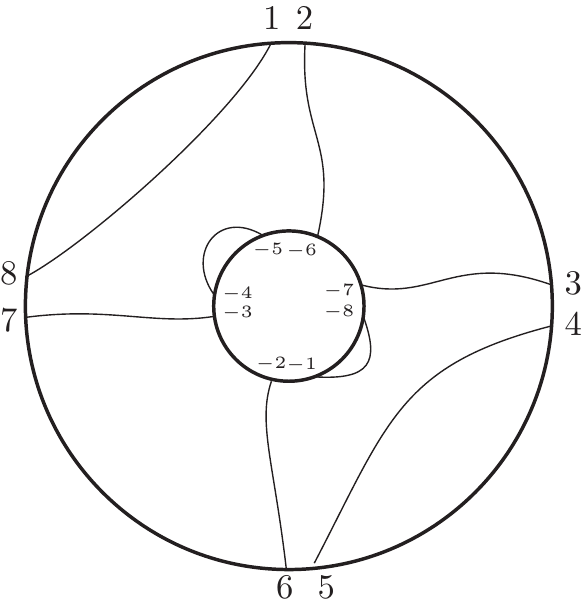} \hfill
\includegraphics{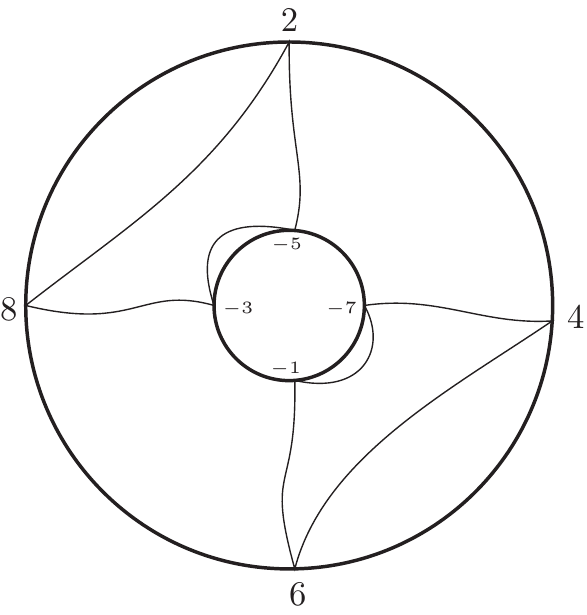}
\caption{\label{fig:permutation_to_permutations}{\small On the left we have $\rho \in NC_2^\delta(8, -8)$ and on the right we have $\sigma = \tilde\omega \rho|_E$. This gives the bijection demonstrated in the proof of Proposition \ref{prop:annular_bijection}. From a topological point of view we squeeze together the pairs of $\tilde\omega = \ab(1, 2) (3, 4) (5, 6) (7,8)\ab (-1, -2) (-3, -4)\ab (-5, -6) \ab(-7, -8)$ to produce the points of $E\ab  = \ab \{2,\ab 4, 6, 8, \ab -1, \ab -3, -5, -7\}$. This gives the embedding of $S_{NC}^\delta(n, \ab-n)$ into $NC_2^\delta(2n ,-2n)$. 
}}
\end{figure}

\begin{proposition}\label{prop:annular_bijection}
\[
\mathop{\sum_{\pi \in \cP_2(2n)}}_{\epsilon\pi\delta\pi\epsilon
\in NC_2^\delta(2n , -2n)}
c^{\#(\omega \vee \pi)}
=
\sum_{\sigma \in S_{NC}^\delta(n, -n)} c^{\#(\sigma)/2}.
\]
\end{proposition}

\begin{proof}
Given $\pi \in \cP_2(2n)$ such that $\rho = \epsilon \pi \delta \pi \epsilon \in NC_2^\delta(2n - 2n)$, let $\sigma = \tilde \omega \rho|_E$. We have shown, Lemma \ref{lemma:the_restrictions} $(b)$,  that $\#(\omega \vee \pi) = \#(\sigma)/2$.

Recall that $\tg^2(E) = E$, so let $\gamma_E = \tg^2|_E$. By Lemma \ref{lemma:the_restrictions} $(c)$ we have that $\#(\sigma) + \#(\sigma^{-1} \gamma_E) = 2n$. With our notation
\[
\gamma_E 
=
(2, 4, 6, \dots, 2n -2, 2n)
(-(2n -1), -(2n - 3), \dots, -3, -1).
\]
Let $\delta_E$ be the pairing of $E$ given by $\delta_E(2 k) = \tilde\omega\delta(2k) = -(2 k - 1)$. For the purposes of this proof we shall use the following notation. Let $S_{NC}^\delta(E)$ be the set of permutations, $\sigma$, of $E$ such that $\#(\sigma) + \#(\sigma^{-1} \gamma_E) = 2n$, $\sigma \delta_E$ is a pairing, and $\sigma$ connects the two cycles of $\gamma_E$. We shall show that the map $\pi \mapsto \sigma$ is a bijection from $\{ \pi \in \cP_2(2n) \mid \epsilon \pi \delta \pi \epsilon \in NC_2^\delta(2n, -2n) \}$ to $S_{NC}^\delta(E)$. 

We shall begin by showing that $\sigma$ connects the cycles of $\gamma_E$. Let $(r, -s)$ be a through cycle of $\rho$. Suppose first that $r$ is even; then by Lemma \ref{lemma.structure.through.pairings} $(c)$, $s$ is also even and so $-(s-1) \in E$. Moreover $\sigma(r) = \tilde\omega(-s) = -(s - 1)$. Thus in this case $\sigma$ connects the two cycles of $\gamma_E$. Next, let us assume that $r$ is odd. Then again by Lemma \ref{lemma.structure.through.pairings} $(c)$, we have $r +1, - s \in E$. Moreover $\sigma(-s) = \tilde\omega(r) = r + 1$. Thus again, $\sigma$ connects the two cycles of $\gamma_E$. We also have
\begin{itemize}
\item
$\delta_E = \epsilon \tilde\omega \epsilon|_E$, and

\item
$\sigma = \epsilon \omega \pi \delta \omega \pi \epsilon|_E$.

\end{itemize}
As 
\[
\tilde\omega \rho \cdot \epsilon \tilde\omega \tilde \epsilon
=
\epsilon \cdot \omega \pi \omega \cdot \delta   \omega \pi \omega \delta \cdot \epsilon
\]
is a pairing, we have that $\sigma\delta_E$ is a pairing. By Lemma \ref{lemma:the_restrictions} we have $\#(\sigma) + \#(\sigma^{-1} \gamma_E) = 2n$, so $\sigma \in S_{NC}^\delta(E)$. 

Conversely let $\sigma \in S_{NC}^\delta(E)$ be given and let $\rho = \sigma^{-1} \tilde \omega \sigma$. As we have conjugated the pairing $\tilde \omega$, $\rho$ is a pairing such that $\tilde \omega \rho|_E = \tilde \omega  \sigma^{-1} \tilde \omega \sigma|_E = \sigma$.   As $\sigma$ connects the cycles of $\gamma_E$, $\rho$ connects the cycles of $\tg$. Next we shall show that $\rho\tg|_E = \sigma^{-1}\gamma_E$ and $\rho\tg|_O = \tilde \omega\sigma \tilde \omega $. This will show that $\#(\rho \tg) = 2n$, so $\rho \in NC_2^\delta(2n, -2n)$. For $k > 0$ we have 
\[
\rho\tg(2k) = \sigma^{-1} \tilde \omega \sigma(2 k + 1) = \sigma^{-1} \tilde \omega(2 k + 1) = \sigma^{-1}(2 k + 1) = \sigma^{-1}\gamma_E(2k),
\]
and 
\begin{align*}
\rho\tg(-(2k-1)) & = \sigma^{-1} \tilde \omega \sigma(-2 k) = \sigma^{-1} \tilde \omega(-2 k) = \sigma^{-1}(-(2 k + 1)) \\
& = \sigma^{-1}\gamma_E(-(2k-1)). 
\end{align*}
This shows that $\rho\tg|_E = \sigma^{-1}\gamma_E$. Again for $k > 0$
\[
\rho\tg(2k-1) = \sigma^{-1} \tilde \omega \sigma(2 k) = \tilde \omega \sigma(2 k) = \tilde \omega \sigma \tilde \omega(2 k -1),
\]
and
\[
\rho\tg(-2k) = \sigma^{-1} \tilde \omega \sigma(-2 k - 1)
= \tilde \omega\sigma(-2 k -1)  = \tilde \omega\sigma\tilde \omega(-2k).
\]

Finally let us show that $\rho = \epsilon \pi \delta \pi \epsilon$ for some pairing $\pi$; or equivalently that for $k > 0$ we have $\epsilon \rho \epsilon(k) < 0$. For $ k > 0$ we have 
\[
\epsilon \rho \epsilon(2k)
=
\epsilon \sigma^{-1} \tilde \omega \sigma (-2k)
=
\epsilon \sigma^{-1} \tilde \omega (-2k)
=
\epsilon \sigma^{-1}(-2k -1) < 0
\]
because if $\sigma^{-1}(-2k -1) < 0$ it must be odd and if $\sigma^{-1}(-2k -1) > 0$ it must be even. Also
\[
\epsilon \rho \epsilon(2k-1)
=
\epsilon \sigma^{-1} \tilde \omega \sigma (2k-1)
=
\epsilon \sigma^{-1} \tilde \omega (2k-1)
=
\epsilon \sigma^{-1}(2k) < 0
\]
because if $\sigma^{-1}(2k) < 0$ it must be odd and if $\sigma^{-1}(2k) > 0$ it must be even. Hence $\rho = \epsilon \pi \delta \pi \epsilon$ where $\pi = \delta \epsilon \rho \epsilon |_{[2n]}$. 
This completes the proof of the claimed bijection. 

If we let $\psi: E \rightarrow [\pm n]$ be the bijection $\psi(2k) = k$ and $\psi(-(2k-1)) = -k$, for $k \in [n]$, we see that $\psi$ conjugates $\gamma_E$ to $\gamma' = (1, \dots, n)(-n, \dots, -1)$, and $\delta_E$ to $\delta|_{[\pm n]}$ and $\sigma$ to a permutation $\sigma' \in S_{\pm n}$ such that
\begin{itemize}

\item
$\sigma'\delta$ is a pairing;

\item
$\#(\sigma') + \#(\sigma'^{-1} \gamma')  = 2n$, and

\item
$\sigma'$ connects the two cycles of $\gamma'$

\end{itemize}
Thus $\sigma' \in S_{NC}^\delta(n, -n)$. Hence
\[
\mathop{\sum_{\pi \in \cP_2(2n)}}_{\epsilon\pi\delta\pi\epsilon
\in NC_2^\delta(2n , -2n)}
c^{\#(\omega \vee \pi)}
=
\sum_{\sigma \in S_{NC}^\delta(n, -n)} c^{\#(\sigma)/2}.
\]
\end{proof}

\begin{remark}
The proof of Proposition \ref{prop:annular_bijection} shows that $|S_{NC}^\delta(n, -n)| = |\{ \rho \in NC_2^\delta(2n -2n) \mid \epsilon \rho\epsilon = \pi \delta \pi$  for some pairing  $\pi \in \cP_2(2n)\}|$. We shall show, in Corollary \ref{cor:exact_formula}, that $|S_{NC}^\delta(n, -n)| = 4^{n-1} - \frac{1}{2} \binom{2n}{n}$, whereas in \cite[Lemma 23]{m}, we showed that $|NC_2^\delta(2n ,-2n)| = \frac{1}{2}( 4^n - \binom{2n}{n})$.  Thus
\[
|NC_2^\delta(2n , -2n)| = 4^{n-1} + |S_{NC}^\delta(n, -n)|.
\]
 
\end{remark}

The conclusions of Theorem \ref{thm:first_main_claim} and Proposition \ref{prop:annular_bijection} give us the following corollary.

\begin{corollary}\label{cor:multi_matrix_moment_theorem}
\begin{eqnarray}\label{eq:multi_matrix_theorem}\lefteqn{
m'_n = \lim_{N \rightarrow \infty}
N\big( \E(\tr(X^n)) - \sum_{\mathclap{\pi \in NC(n)}} c^{\#(\pi)} \big)}\notag \\ 
&= & 
\sum_{\mathclap{\pi \in NC(n)}} c' \#(\pi) c^{\#(\pi)-1}
+ \
\sum_{\mathclap{\sigma \in S_{NC}^\delta(n, -n)}} c^{\#(\sigma)/2}.
\end{eqnarray}

\end{corollary}

\begin{proposition}\label{prop:counting_diagrams}
	\begin{align*}
		\sum_{\mathclap{\sigma \in S_{NC}^\delta(n, -n)}} 
		  c^{\#(\sigma)/2} 
	=
		\sum_{l = 1}^{\floor{n/2}} \sum_{j = 0}^{n-2l} \binom{n}{j} \binom{n}{j+2l} c^{j+l}  
	=
 		\sum_{k=1}^{n-1} a_{k} c^{k}
	\end{align*}
where $a_{k} = \ds\sum_{i=1}^{\mathclap{\min \{k,n-k\}}} \quad \binom{n}{k-i}\binom{n}{k+i} =\frac{1}{2}\left( \binom{2n}{2k} - \binom{n}{k} \binom{n}{k} \right) $. 
\end{proposition}
\begin{figure}
\includegraphics{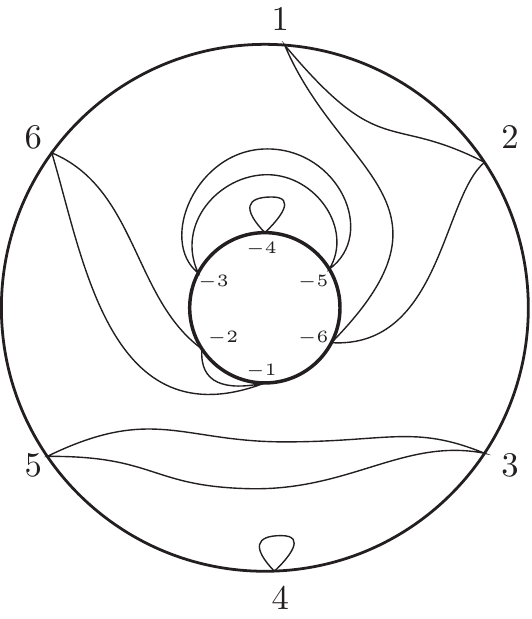} \hfill
\includegraphics{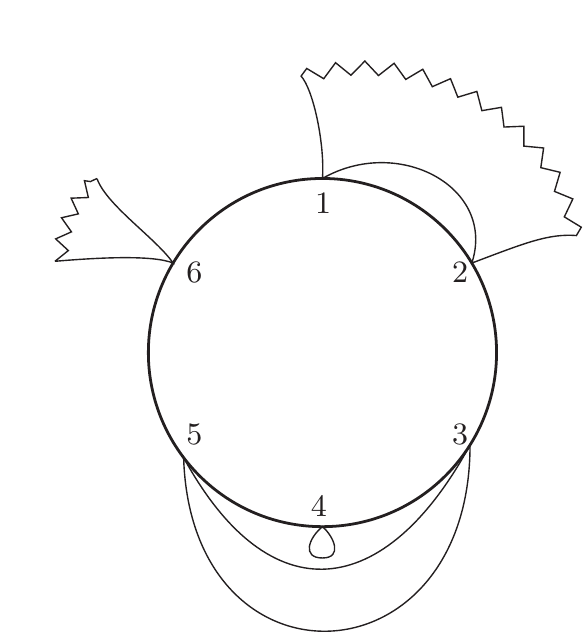}
\caption{\label{fig:half_permutations}{\small On the left we have an element of $S_{NC}^\delta(6, -6)$. We convert it to a \textit{non-crossing circular half-permutation} (in the language of \cite[\S 6 Def.~2]{kms})}, which is shown on the right. Then \textit{open} blocks are those with a zigzag edge; the \textit{closed} blocks have a smooth edge. We will always get an even number of open blocks. See \cite[\S 6 Def.~2]{kms} for details. We shall use this in the proof of Proposition \ref{prop:counting_diagrams}.}
\end{figure}
\begin{proof}
	The condition $\delta \sigma \delta = \sigma^{-1}$ for any element  $\sigma \in S_{NC}^\delta(n, -n)$ implies that the number of through cycles and the number of non-through cycles in in $\sigma$ are both even numbers. 
	Thus, since each $\sigma \in S_{NC}^\delta(n, -n)$ must have at least two through cycles, the sum $	\sum_{\sigma \in S_{NC}^\delta(n, -n)} c^{\#(\sigma)/2} $ can be rewritten as 
	\[
	\sum_{l = 1}^{\floor{n/2}} \sum_{j = 0}^{n-2l}
	 	\abs{ S_{NC}^\delta(n, -n)_{j,l}} c^{j+l} 
	\]
	where $S_{NC}^\delta(n, -n)_{j,l}$ denotes the set of all permutations from $S_{NC}^\delta(n, -n)$ with exactly $2j$ non-through cycles and $2l$  through cycles. 
	The condition $\delta \sigma \delta = \sigma^{-1}$ also implies that each permutation $\sigma \in S_{NC}^\delta(n, -n)_{j,l}$ is completely and uniquely determined by the non-crossing circular half-permutation with $j$ closed blocks and $2l$ open blocks resulting from  restricting $\sigma$ to the set $[n]$, see Figure \ref{fig:half_permutations}. 
	It then follows from \cite[Theorem 25]{kms} that $\abs{ S_{NC}^\delta(n, -n)_{j,l}}$ is given by $\binom{n}{j} \binom{n}{j+2l}$, and hence, we obtain 
	\[
	\sum_{\sigma \in S_{NC}^\delta(n, -n)} c^{\#(\sigma)/2} 
	=
	\sum_{l = 1}^{\floor{n/2}} \sum_{j = 0}^{n-2l} \binom{n}{j} \binom{n}{j+2l} c^{j+l}. 
	\]
See \cite[Figure 16]{kms} for an explicit example of how the counting works.
Applying the change of variable $k = j + l$ in the sum in the right-hand side of the equality above, and regrouping its terms with respect to $c^k$, yields  
	\[
	\sum_{l = 1}^{\floor{n/2}} \sum_{j = 0}^{n-2l} \binom{n}{j} \binom{n}{j+2l} c^{j+l} 
	=
	\sum_{k=1}^{n-1} c^{k} \sum_{l=1}^{\min \{k,n-k\}} \binom{n}{k-l}\binom{n}{k+l}. 
	\] 
Note that the sum $\sum_{l=1}^{\min \{k,n-k\}} \binom{n}{k-l}\binom{n}{k+l}$ is invariant under the transformation $k' = n -k$, so in evaluating the sum we may assume that $k \leq n - k$. Then we have
\begin{align*}\lefteqn{
\sum_{l=1}^k \binom{n}{k-l}\binom{n}{k+l}
\stackrel{(k-l) \mapsto l}{=}
\sum_{l=0}^{k-1}\binom{n}{l}\binom{n}{2k-l} } \\
&= \frac{1}{2} \Big(
\sum_{l=0}^{k-1}\binom{n}{l}\binom{n}{2k-l}
+
\sum_{l=k+1}^{2k} \binom{n}{2k-l}\binom{n}{l} \Big) \\
& =
\frac{1}{2} \Big( \sum_{l=0}^{2k}\binom{n}{l}\binom{n}{2k-l} -
\binom{n}{k}^2 \Big) 
=
\frac{1}{2} \Big( \binom{2n}{2k} - \binom{n}{k}^2 \Big)
\end{align*}
where the last equality is by Vandermonde convolution (see e.g. \cite[Eq. (5.22)]{gkp}).

\end{proof}

\section{The case of independent Wishart matrices}
\label{section:independent_wishart_matrices}

In this section we extend Corollary \ref{cor:multi_matrix_moment_theorem} to the case of a family of independent Wishart matrices. The conclusion is exactly what happens at the first order; the blocks of $\pi$, in the first term, or $\sigma$ in the second can only connect a matrix with itself, see \cite[Cor. 9.4]{mn}. 

So to this end let $X_{1, N}, \dots, X_{s, N}$ be $s$ independent Wishart matrices as in Definition \ref{def:real_wishart}. Let $l_1, \dots, l_n \in [s]$ and $\ker(l) \in \cP(n)$ be the kernel of $l$. We know from \cite[Cor. 9.4]{mn} that 
\begin{equation*}
\lim_N \E(\tr(X_{l_1, N} \cdots X_{l_n, N}))
=
\mathop{\sum_{\pi \in NC(n)}}_{\pi \leq \ker(l)} c^{\#(\pi)}. 
\end{equation*}
Given any partition $\tau$ of $[n]$ we get a partition $\tilde\tau$ of $[\pm n]$ by setting, for $r, s \in [\pm n]$, $r \sim_{\tilde\tau} s$ if and only if $|r| \sim_\tau |s|$. Here, $|r|$ denotes the absolute value of $r$. 

We now turn to the multi-matrix version of Equation (\ref{equation:1}) of Lemma \ref{lemma:1}. Given our $l_1, \dots, l_n \in [s]$, let $k_1, k_2, \dots, k_{2n-1}, k_{2n} \in [s]$ be given by $k_{2r-1} = k_{2r} = l_r$ for $1 \leq r \leq n$. Then $\ker(k) \in \cP(2n)$. By repeating the proof of Lemma \ref{lemma:1} we have
\begin{equation}\label{equation:2}
\E(\tr(X_{l_1,N} \cdots X_{l_n, N})) =  \,
\mathop{\sum_{\mathclap{\pi \in \cP_2(2n)}}}_{\mathclap{\pi \leq \ker(k)}} 
\ \ \Big(\frac{M}{N}\Big)%
^{\#(\omega  \vee  \pi )}
N^{\#(\gamma \omega \gamma^{-1} \vee \pi) + \#(\omega  \vee  \pi ) -(n+1)}.
\end{equation}

The following Lemma will sort out which $\sigma$'s can appear in the second term on the right hand side of Equation (\ref{eq:multi_matrix_theorem}).

\begin{lemma}\label{lemma:kernel_equivalence}
Suppose $\pi \in \cP_2(2n)$ and $\rho = \epsilon \pi \delta \pi \epsilon \in NC_2^\delta(2n ,-2n)$. Suppose in addition that $\sigma \in S_{NC}^\delta(n, -n)$ is the permutation produced in the bijection of Proposition \ref{prop:annular_bijection}. Then
\[
\sigma \leq \widetilde{\ker(l)} \Leftrightarrow \pi \leq \ker(k).
\]
\end{lemma}

\begin{proof}
Suppose first that $\sigma \leq \widetilde{\ker(l)}$, and we will show that $\pi \leq \ker(k)$; i.e. if $(r, s) \in \pi$ then $k_r = k_s$. We break this into two cases.

\medskip\noindent\textit{Case $(a)$}: $(r, -s)$ is a through string of $\rho = \epsilon \pi \delta \pi \epsilon$. Then by Lemma  \ref{lemma.structure.through.pairings} $(d)$, $r$ and $s$ have the same parity. In the even case we have $\sigma(r/2) = s/2$ and thus $k_s =  l_{s/2} = l_{\sigma(r/2)} = l_{r/2} = k_r$. In the odd case we have $\sigma(-(s+1)/2) = (r + 1)/2$. Hence $k_r = k_{r+1} = l_{(r+1)/2} = l_{(s + 1)/2} = k_{s+1} = k_s$.
In either case $k_r = k_s$.

\medskip\noindent\textit{Case $(b)$}: $(r, s)$ is a pair of $\rho = \epsilon \pi \delta \pi \epsilon$. Then by Lemma \ref{lemma.structure.through.pairings}, $r$ and $s$ have opposite parities; suppose $r$ is even. Then $\sigma(r/2) = \sigma((s + 1)/2)$. Hence $k_r = l_{r/2} = l_{(s +1)/2} = k_{s + 1} = k_s$.

In the opposite direction, suppose that $\pi \leq \ker(k)$, and we will show that $\sigma \leq \widetilde{\ker(l)}$. Suppose that $r, s \in [n]$.

\medskip\noindent\textit{Case $(c)$}: $\sigma(r) = s$. Then $\rho(2r) = 2s-1$; so $\pi(2r) = 2s -1$. Hence $l_r = k_{2r} = k_{2s-1} = k_{2s} = l_s$. 

\medskip\noindent\textit{Case $(d)$}: $\sigma(r) = -s$. Then $\rho(2r) = -2s$; so $\pi(2r) = 2s$. Hence $l_r = k_{2r} = k_{2s} = l_s$. With these two cases we conclude that $\sigma \leq \widetilde{\ker(l)}$ as claimed. 
\end{proof}

\begin{theorem}\label{thm:main_theorem_multi_matrix_case}
Suppose $X_{1, N}, \dots , X_{s, N}$ are independent Wishart matrices with the same shape, i.e. all are obtained using the same $M$. Suppose $\lim_N (M - cN) = c'$. Then
\begin{align*}\lefteqn{
\lim_N 
N\Big\{ \E(\tr(X_{l_1, N} \cdots X_{l_n,N})) - 
\mathop{\sum_{\pi \in NC(n)}}_{\pi \leq \ker(l)} c^{\#(\pi)}\Big\}} \\
&=
\mathop{\sum_{\pi \in NC(n)}}_{\pi \leq \ker(l)} c' \#(\pi) c^{\#(\pi)-1}
+ 
\mathop{\sum_{\sigma \in S_{NC}^\delta(n, -n)}}_{\sigma \leq \widetilde\ker(l)}
c^{\#(\sigma)/2}.
\end{align*}
\end{theorem}

\begin{proof}
From Equation (\ref{equation:2}) we only have to decide which $\pi$'s survive in the large $N$ limit. According to Theorem \ref{thm:first_main_claim} there are two cases: the non-crossing $\pi$'s and those for which $\epsilon \pi \delta \pi \epsilon \in NC_2^\delta(2n, -2n)$. That the non-crossing ones produce the first term is proved in \cite[Cor. 9.4]{mn}. In the second case we use the bijection in Proposition \ref{prop:annular_bijection} restricted to those $\pi$'s such that $\pi \leq \ker(k)$.  The image of this subset is provided by Lemma \ref{lemma:kernel_equivalence}. This produces the second term on the right hand side of the statement.
\end{proof}

\begin{remark}
We can lift the assumption that all the Wishart matrices $X_{1, N}, \dots , X_{s, N}$ have the same shape parameters $c$ and $c'$. For each block of $\pi$ or $\sigma$ we simply use corresponding pair $(c, c')$ of shape parameters. This suggests that these matrices exhibit a new kind of independence, in fact the same as found in \cite[Thm. 37]{m}. We will address this point in a subsequent paper.
\end{remark}

\section{preliminaries for the recursion formulas for $|S_{NC}^\delta(n, -n)|$}\label{sec:first_recursion}

In this section we will present some preliminaries for the recursion formula for $|S_{NC}^\delta(n, -n)|$ given below. In Proposition \ref{prop:left_or_right} we give a criterion for dividing $S_{NC}^\delta(n, -n)$

The formula will be proved in Theorem \ref{thm:the_recursion} of Section \ref{sec:the_recursion}. In the formula $m_n = |NC(n)|$ and $\ols{m}'_n = |S_{NC}^\delta(n, -n)|$. 
\[
\ols{m}'_n = (n-1) m_{n-1} + 2 \sum_{k=1}^{n-1} m_{k-1}\,  \ols{m}'_{n-k}.
\]
This is the same recursion as the one for the areas under Dyck paths, see \cite{msv} or \cite[Lemma 3.12]{f}. While the formula above is a special case of 
Equation (\ref{eq:mp_recurrence}), we need the results here to obtain (\ref{eq:mp_recurrence}). When we put $c = 1$ in Figure \ref{table:the_first_six_moments} above we get for $n = 1, \dots, 7$, the numbers 0, 1, 6, 29, 130, 562, 2380. This is sequence A\,008\,549 in \cite{sl}. 

\begin{notation}\label{notation:tilde_gamma}
Let $n \geq 1$ be fixed. Let us recall our notation: $\gamma = (1,\ab  2, 3, \dots, n) \in S_{n} \subset S_{\pm n}$. Let $\delta \in S_{\pm n}$ be given by $\delta(k) = - k$. $S_{NC}(n, -n)  = \{ \pi \in S_{\pm n} \mid  \pi$ connects the cycles of $\gamma \delta \gamma^{-1} \delta$ and $\#(\pi) + \#(\pi^{-1} \gamma \delta \gamma^{-1} \delta) = 2 n \}$. $S_{NC}^\delta(n, -n) = \{ \pi \in S_{NC}(n, -n) \mid \pi\delta$ is a pairing $\mbox{}\}$. For $\pi \in S_{NC}^\delta(n, -n)$, let $K(\pi) = \delta \gamma^{-1} \delta \pi^{-1}\gamma$ be the \textit{Kreweras complement} of $\pi$. Note that $K(\pi) \in S_{NC}^\delta(n,\ab -n)$. This is the same Kreweras complement as used in \cite{r1}. 
\end{notation}

\begin{lemma}\label{lemma:unrolling_the_annulus}
Suppose, $\pi \in S_{NC}^\delta(n, -n)$,  $1 \leq j < k \leq n$, and $\pi^{-1}(j) = -k$. Let 
\[
\hg= \gamma \delta \gamma^{-1} \delta \cdot (-j, \gamma^{-1}(k))\, \cdot (-k, \gamma^{-1}(j)).
\]
Then 
\[
\hg =
(1, 2, \dots, j-1, -(k-1), -(k-2), \dots, -j, k, k+1, \dots, n)
\]
\[
\mbox{}\times
(-n, -(n-1), \dots, -k, j, j+1, \dots, k-1, -(j-1), \dots, -1),
\]
and $\pi$ is non-crossing with respect to $\hg$. By this we mean that each cycle of $\pi$ is contained in one of the two cycles of $\hg$ and $\#(\pi) + \#(\pi^{-1}\hg) = 2n + 2$. 
\end{lemma}

\begin{proof}
\begin{align*}\lefteqn{
\hg =
 \gamma \delta \gamma^{-1} \delta (-j, \gamma^{-1}(k)) (-k, \gamma^{-1}(j))}
\\
& =  \gamma \cdot (-\gamma^{-1}(j), \gamma^{-1})k)  \cdot 
(-\gamma^{-1}(k), \gamma^{-1}(j))  \cdot \delta \gamma^{-1} \delta
\end{align*}
so
\[
\#(\pi^{-1} \hg) 
= \#(
(-\gamma^{-1}(k), \gamma^{-1}(j)) \cdot  \delta \gamma^{-1} \delta \cdot 
\pi^{-1} \gamma (-\gamma^{-1}(j), \gamma^{-1}(k)))
\]
\[
=
\#( (-\gamma^{-1}(k), \gamma^{-1}(j)) \cdot K(\pi) \cdot (-\gamma^{-1}(j), \gamma^{-1}(k)) ).
\]
Since $\pi^{-1}\hg (\gamma^{-1}(j)) = -k$ and $\pi^{-1}\hg (\gamma^{-1}(k)) = -j$ we have that $ -\gamma^{-1}(k)$ and $\gamma^{-1}(j) $ are in the same cycle of $K(\pi)$ and $ \gamma^{-1}(k)$ and $-\gamma^{-1}(j) $ are in the same cycle of $K(\pi)$; but not in the same cycle as $ -\gamma^{-1}(k)$ and $\gamma^{-1}(j) $. Hence
\[
\#( (-\gamma^{-1}(k), \gamma^{-1}(j)) \cdot K(\pi) \cdot (-\gamma^{-1}(j), \gamma^{-1}(k)) ) = \#(K(\pi)) + 2.
\]
Thus 
\[
\#(\pi) + \#(\pi^{-1} \hg) = \#(\pi) + \#(K(\pi)) + 2 = 2n + 2.
\]
Let $\langle \pi, \hg \rangle$ be the subgroup generated by $\pi$ and $\hg$. If $\langle \pi, \hg \rangle$ acts transitively on $[\pm n]$ then there is an integer $g \geq 0$ such that
\[
\#(\pi) + \#(\pi^{-1} \hg) + \#(\hg) = 2n + 2(1 - g).
\]
Then we would have
\[
2m + 2 = \#(\pi) + \#(\pi^{-1} \hg) = 2n - 2g,
\]
which is impossible. Hence $\langle \pi, \hg \rangle$ does not act transitively on $[\pm n]$, this means that no cycle of $\pi$ meets both of the cycles of $\hg$. Hence each cycle of $\pi$ is contained in one of the cycles of $\hg$.  
\end{proof}

\begin{lemma}\label{lemma:the_converse}
Suppose $1 \leq j < k \leq n$ and $\hg$ is as in Lemma \ref{lemma:unrolling_the_annulus}. Let $\pi \in S_{\pm n}$ be such that $(i)$ $\pi \delta$ is a pairing; $(ii)$ each cycle of $\pi$ is contained in one of the two cycles of $\hg$, $(iii)$ $\#(\pi) + \#(\pi^{-1}\hg) = 2n + 2$, and $(iv)$ $\pi \vee \gamma \delta \gamma^{-1} \delta = 1_{\pm n}$. Then $\pi \in S_{NC}^\delta(n, -n)$.
\end{lemma}

\begin{proof}
Let us consider the cycle of $\hg$ containing 1:
\[
(1, 2, \dots, (j-1), -(k-1), -(k-2), \dots, -j, k, k+1, \dots, n).
\] 
If $-j$ and $\gamma^{-1}(j)$ are in the same block of $\pi^{-1}\hg$ then $\pi$ cannot have a cycle connecting a point in $\{-(k-1), -(k-2), \dots, -j\}$ to points outside, violating $(iv)$. Likewise for $-k$ and $\gamma^{-1}(k)$, they cannot be in the same block of $\pi^{-1}\hg$. In addition $-j$ and $\gamma^{-1}(j)$ are in the cycle of $\hg$ above and $-k$ and $\gamma^{-1}(k)$ are in the other cycle of $\hg$:
\[
(-n, -(n-1), \dots, -k, j , j + 1, \dots, k-1, -(j-1), -(j-2), \dots, -1). 
\]
So we have that both pairs $(-j, \gamma^{-1}(k))$ and $(-k, \gamma^{-1}(j))$ are in different cycles of $\pi^{-1}\hg$. Hence 
\begin{align*}
\#(K(\pi)) = \#(\pi^{-1} \gamma \delta \gamma^{-1} \delta)
&= 
\#(\pi^{-1} \hg \cdot (-j, \gamma^{-1}(k)) \cdot (-k, \gamma^{-1}(j)) \\
& =
\#(\pi^{-1} \hg) - 2.
\end{align*}
Thus $\#(\pi) + \#(K(\pi)) = \#(\pi) + \#(\pi^{-1} \hg) -2 = 2n$. Hence $\pi \in S_{NC}(n, -n)$. As we also have $(i)$ we get that $\pi \in S_{NC}^\delta(n, -n)$. 
\end{proof}

\begin{proposition}\label{prop:left_or_right}
Let $\pi \in S_{NC}^\delta(n, -n)$ and suppose $\pi^{-1}(1) = j \in [m]$. Let $k \geq 2$ be the smallest integer such that $\pi^{-1}(k) \in [-n]$. 
Let $I_1 = \{1, 2, \dots, j\}$ and $I_2 = \{ j+1, \dots, n\}$. Then either $I_1$ or $I_2$, but not both, meets a through cycle of $\pi$.
\end{proposition}

\begin{proof}
Let $l = -\pi^{-1}(k)$. Note that $\pi^{-1}(l) = -k$; so $k < l$. As in Notation \ref{notation:tilde_gamma}, let $\hg$ be the permutation
\[
(1, 2, \dots, k-1, -(l-1), \dots, -k, l, l+1, \dots, n)
\]
\[
\mbox{} \times
(-n, \dots, -(l+1), -l, k, \dots, l-1, -(k-1), \dots, -1).
\]
Let us consider the cycle of $\hg$ containing $1$:
\[
(1, 2, \dots, k-1, -(l-1), \dots, -k, l, l+1, \dots, n);
\]
the cycle of $\pi$ containing 1 must, by Lemma \ref{lemma:unrolling_the_annulus}, be contained in this cycle, and in particular we have
\[
j \in
(1, 2, \dots, k-1, -(l-1), \dots, -k, l, l+1, \dots, n).
\]
So either $j \leq k -1$ or $j \geq l$.

Suppose $j \leq k-1$. Then $I_1 \subset \{ 1, 2, \dots, k-1\}$ and so no through block meets $I_1$ because the first through block is at $k$; but then $k$ is in $I_2$, so a through block meets $I_2$. On the other hand, suppose $j \geq l$. Then $I_2 \subset \{l+1, \dots, n\}$. If a through block of $\pi$ were to meet $I_2$, then the whole cycle would, by Lemma \ref{lemma:unrolling_the_annulus}, lie in $I_2$; but this impossible because $I_2 \subset [m]$. However $k < l \leq j$ so $k \in I_1$ and thus $I_1$ does meet a through block. 
\end{proof}

\begin{figure}
\setbox1=\hbox{\includegraphics{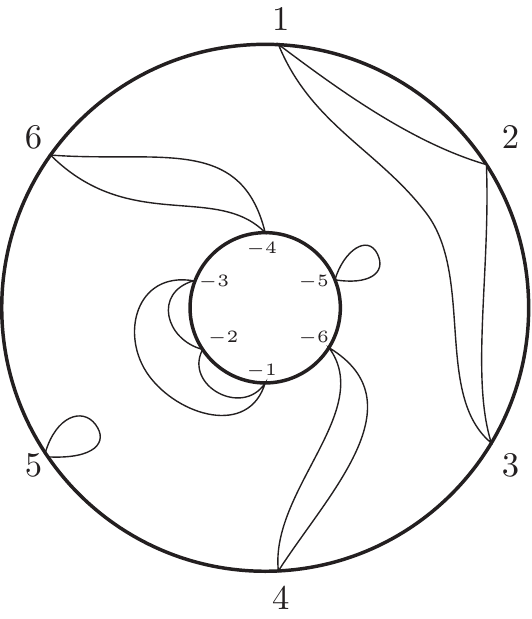}}
\setbox2=\hbox{\includegraphics{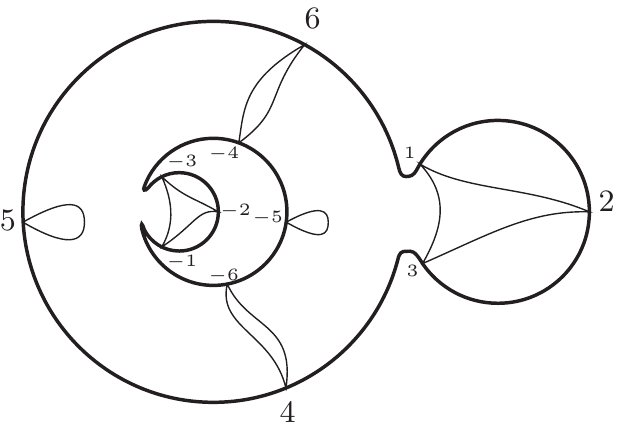}}
\begin{center}\leavevmode
$\vcenter{\hsize=\wd1\box1}$ \hfill
$\vcenter{\hsize=\wd2\box2}$
\end{center}

\caption{\small\label{fig10:intermediate_pinch} On the left we have a partition in $S_{NC}^\delta(6, -6)$, with $\pi^{-1}(1) = 3$ and $\pi(-1) = -3$.  We shrink to a point the line joining $1$ to $\pi^{-1}(1)$ and the line joining $-1$ to $\pi(1)$. This breaks off two circles. We will label $\pi_1$ the non-crossing partition of the circle we have produced after we remove $\pi^{-1}(1)$ from the block containing $1$. This shrinking procedure  shows a step half way between the left and right figures in Figure \ref{fig3:the_second_type} \textit{infra}.
\bigskip\bigskip\hbox{}}
\end{figure}

\section{the subsets $S_\mathrm{I}$, $S_{\mathrm{II}}$, and $S_{\mathrm{III}}$}\label{sec:the_subsets}

We shall write $S_{NC}^\delta(n, -n)$ as the disjoint union of three subsets according to the value of $\pi^{-1}(1)$. By Proposition \ref{prop:left_or_right}, for each $\pi \in S_{NC}^\delta(n, -n)$ exactly one of the three cases $(a)$, $(b)$, or $(c)$ below holds. 

\begin{definition}\label{def:the_division}
Let 
\begin{enumerate}

\item
$S_I = \{ \pi \in S_{NC}^\delta(n, -n) \mid \pi^{-1}(1) \in [-n]\}$;

\item
$S_{\mathit{II}} = \{ \pi \in S_{NC}^\delta(n, -n) \mid \pi^{-1}(1) \in [n]$ and the interval $I_1 = [1, \pi^{-1}(1)]$ does \textbf{not} meet a through cycle of $\pi\}$;

\item
$S_{\mathit{III}} = \{ \pi \in S_{NC}^\delta(n, -n) \mid \pi^{-1}(1) \in [n]$ and the interval $I_1 = [1, \pi^{-1}(1)]$ does meet a through cycle of $\pi\}$.
\end{enumerate}
Let 
\begin{enumerate}
\setcounter{enumi}{3}
\item
$T_I = \{ (k, \pi) \mid k \in [n-1]$ and  $\pi \in NC(n-1)\}$;

\item
$T_{\mathit{II}} = \bigcup_{k=1}^{n-1} NC(k-1) \times S_{NC}^\delta(n-k, -(n-k))$;

\item
$T_{\mathit{III}} = \bigcup_{k=2}^{n} S_{NC}^\delta(k-1, -(k-1)) \times NC(n-k)$;.
\end{enumerate}
\end{definition}

To simplify the notation, we have adopted the convention that the cardinality of $NC(0)$ is $1$. On the other hand, the cardinality of $S_{NC}^\delta(1,-1)$ is $0$. We shall exhibit bijections from $S_I$ to $T_I$, from $S_{\mathit{II}}$ to $T_{\mathit{II}}$, and from $S_{\mathit{III}}$ to $T_{\mathit{III}}$ in Notations \ref{remark:T_I}, \ref{remark:T_II}, and \ref{remark:T_III} respectively. In each case, $k = \pi^{-1}(1)$.

\begin{remark}\label{remark:T_I}
Let $k \in [n-1]$ and $\sigma \in NC(n-1)$ be given. After relabelling we consider $\sigma$ to be a non-crossing partition of the $n-1$ points $\{1, 2, 3, \dots, k-1, -n, -(n-1), \dots, -(k+1)\}$. By this we mean we relabel the point according to the  map:
\[
j \mapsto \begin{cases} j & 1 \leq j < k \\
             -(n + k - j) & k \leq j \leq n-1
           \end{cases}.
\]
We let $\pi_1$ be the partition of $\{1, 2, 3, \dots, k-1, -n, -(n-1), \dots, -(k+1), -k\}$ obtained by joining $-k$ to the block containing $1$. Since $1$ and $-k$ are cyclically adjacent $\pi_1$ will be non-crossing. Then we let $\pi = \pi_1 \delta \pi^{-1}_1 \delta \in S_{\pm n}$. By construction $\pi^{-1}(1) = -k$. Also, as the two cycles of $\hg$ are
\[
c_1 = (1, 2, 3, \dots, k-1, -n, -(n-1), \dots, -(k+1), -k) \mathrm{\ and }
\]
\[
c_2 = (k, k+1, \dots, n, -(k-1), -(k-2), \dots, -2, -1)
\]
we have that $\pi$ is non-crossing with respect to $\hg$. 

Note that $\pi^{-1} \hg(-k) = \pi^{-1}(1) = \pi_1^{-1}(1) = -k$ and $\pi^{-1} \hg(-1) = \pi^{-1}(k) = \delta\pi_1\delta(k) = \delta\pi_1(-k) = -1$. So $-k$ and $-1$ are singletons of $\pi^{-1}\hg$. Also
$k-1$ and $m$ are  in different cycles of $\pi^{-1}\hg$. Thus the four points $-1$, $-k$, $k-1$ and $m$ are all in different cycles of $\pi^{-1}\hg$.   Thus $\#(\pi^{-1} \gamma\delta \gamma^{-1} \delta) = \#(\pi^{-1} \hg) - 2$. Hence $\#(\pi) + \#(K(\pi)) = 2n$. 
\end{remark}

\begin{figure}
\begin{center}
\includegraphics{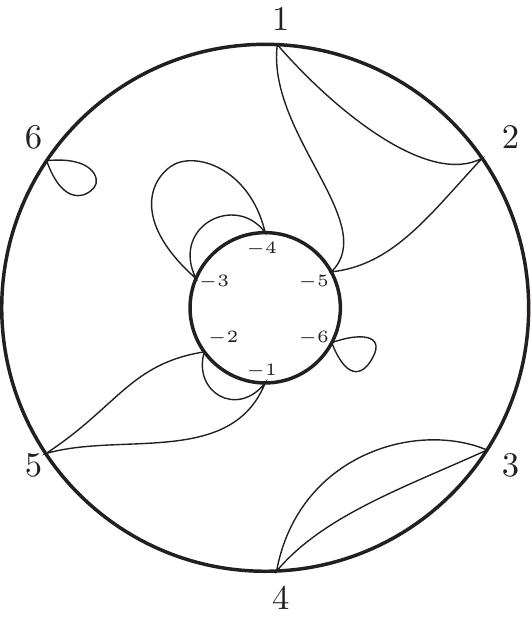} \hfill
\includegraphics{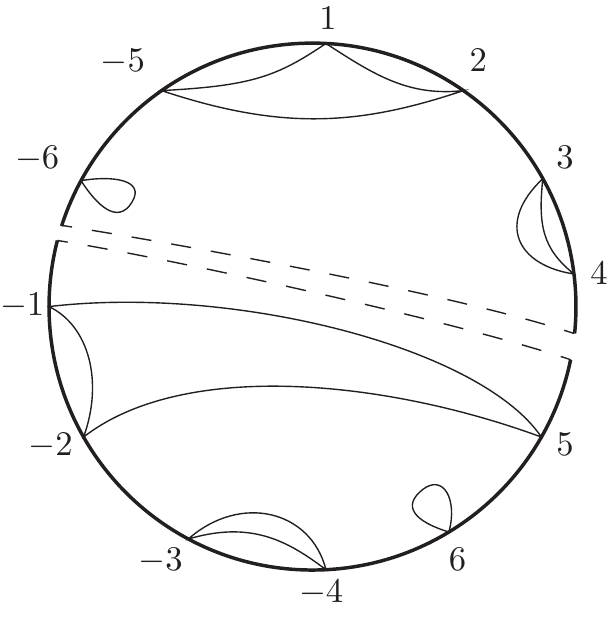} 

\kern-1.5em

\includegraphics{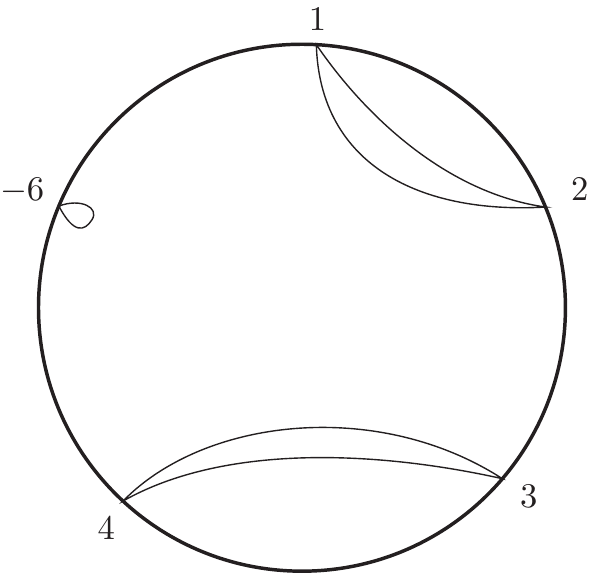}
\end{center}

\caption{\small\label{fig2:the_first_type} On the top left we have $\pi = (1, 2, -5)\ab(3, 4)\ab(-1, 5, -2)\ab(-3, -4)(6)(-6) \in S_{NC}^\delta(6, -6)$. In the terminology of Definition \ref{def:the_division}, $\pi \in S_I$. On the top right we have cut open the annulus, the cycles of $\pi$ form a non-crossing partition of this disc. In the notation of Lemma \ref{lemma:unrolling_the_annulus}, $\hg = (-6, -5, \ab 1, 2, \ab 3, \ab 4)(5, 6, -4, -3, -2, -1)$. We remove $\pi^{-1}(1)$ from the block containing $1$ and $\pi(-1)$ from the block containing $-1$ to obtain two copies of the figure on the bottom, the second copy (not shown) is the mirror image of the first. In the notation of Definition \ref{sec:the_subsets} $(d)$ and Theorem \ref{thm:T_I_bijection}, $k = 5$ and $\sigma = (1, 2)(3, 4)(-6)$, shown in the lower figure.}
\end{figure}
\begin{theorem}\label{thm:T_I_bijection}
The map in Remark \ref{remark:T_I} is a bijection from $T_I$ to $S_I$.
\end{theorem}

\begin{proof}
We need to give the inverse map, see Figure \ref{fig2:the_first_type}. Given $\pi \in S_{NC}^\delta(n,\ab -n)$ we let $k = -\pi^{-1}(1)$. By Lemma \ref{lemma:unrolling_the_annulus}, with $j = 1$, we know that each cycle of $\pi$ is contained in one of the two cycles of $\hg$. 
Let $\pi_1$ be the cycles of $\pi$ in the cycle $c_1$ of $\hg$:
\[
(1, 2, 3, \dots, k-1, -n, -(n-1), \dots, -(k+1), -k).
\]
By construction $1$ and $-k$ are in the same block of $\pi_1$. Let $\sigma$ be the partition of $\{1, 2, 3, \dots, k-1, -n, -(n-1), \dots, -(k+1)\}$ obtained from $\pi_1$ by removing $-k$ from the block containing $1$. Then $\sigma$ is a non-crossing partition in $NC(\{1, 2, 3, \dots, k-1, -n, -(n-1), \dots, -(k+1)\})$ which we identify with $NC(m-1)$, using the inverse of the labelling map above. So from this we get the pair $(k, \sigma) \in T_I$. And by applying the construction in Notation \ref{remark:T_I} we get back to $\pi$. Thus the map is a bijection. 
\end{proof}

\begin{remark}\label{remark:T_II}
Given $\pi \in S_{NC}^\delta(n, -n)$ with $\pi^{-1}(1) \in [n]$. Let $j = \pi^{-1}(1)$ and $k$ be as in Proposition \ref{prop:left_or_right}, i.e. $k > 1$ is the smallest integer such that $\pi^{-1}(k) \in [-n]$. Suppose that we are in the case where $I_1 = \{1, 2, \dots, j\}$ does not meet a through block of $\pi$, but $I_2$ does. So $1 \leq j < n-1$. Thus $\pi|_{I_1}$ is a non-crossing partition of $I_1$, and $1$ and $j$ are in the same block of $\pi|_{I_1}$. Let $\pi_1 \in NC(j-1)$ be the partition obtained from $\pi|_{I_1}$ obtained by removing $j$ from the block containing $1$. Also $\delta \pi_1^{-1} \delta$ is a non-crossing partition of $[-(j-1)]$. Let $\pi_2 = \pi|_{I_2 \cup -I_2}$. Since this a restriction of a non-crossing annular permutation, it is itself a non-crossing annular permutation of $(I_2, -I_2)$. If we identify the points of $I_2$ with $[n - j]$ we have $\pi_2 \in S_{NC}(n-j, -(n-j))$. Since $\pi\delta$ is a pairing $\pi_2 \delta|_{I_2}$ is also a pairing.  Thus $\pi_2 \in S_{NC}(n-j, -(n-j))$. Hence we get a pair $(\pi_1, \pi_2) \in NC(j-1) \times S_{NC}^\delta(n-j, -(n-j))$. For $j > 1$ we have $\#(\pi) = 2\#(\pi_1) + \#(\pi_2)$. When $j =1$ something special happens: $(1)$ and $(-1)$ are singletons of $\pi$. Hence $\pi_1$ is the unique partition of the empty set. In this case $\#(\pi) = \#(\pi_2) + 2$. If $j = n-1$ or $n$ then, as we are assuming that $I_1$ does not meet a through block,  $\pi$ could not have a through block. As this is not possible the largest possible value of $j$ is $n - 2$
\end{remark}

\begin{theorem}
The map in Remark \ref{remark:T_II} is a bijection from $T_{\mathit{II}}$ to $S_{\mathit{II}}$.
\end{theorem}

\begin{proof}
We need to give the inverse map, see Figure \ref{fig3:the_second_type}. Let $(\pi_1, \pi_2) \in NC(j-1) \times S_{NC}^\delta(n-j, -(n-j))$ be given. If we identify the points of $I_2$ with $[n - j]$ we have $\pi_2 \in S_{NC}(I_2, -I_2)$. We add $j$ to block of $\pi_1$ containing $1$ and denote this non-crossing partition $\tilde\pi_1$. Then, set $\pi = \tilde\pi_1 \delta \tilde\pi_1^{-1}\delta \pi_2$. Because we have inserted $\tilde\pi_1$ into the gap between $n$ and $j+1$, we get a non-crossing annular permutation; similarly with $\delta \tilde\pi_1^{-1} \delta$. Hence $\pi \in S_{NC}(n, -n)$. Since $(\tilde\pi_1\delta \tilde\pi_1^{-1}\delta) \delta = \tilde\pi_1 \delta \tilde\pi_1^{-1}$ is a pairing we have that $\pi \delta$ is a pairing and $\pi \in S_{NC}^\delta(n, -n)$.
\end{proof}

\begin{figure}\noindent
\includegraphics{fig4} \hfill
\includegraphics{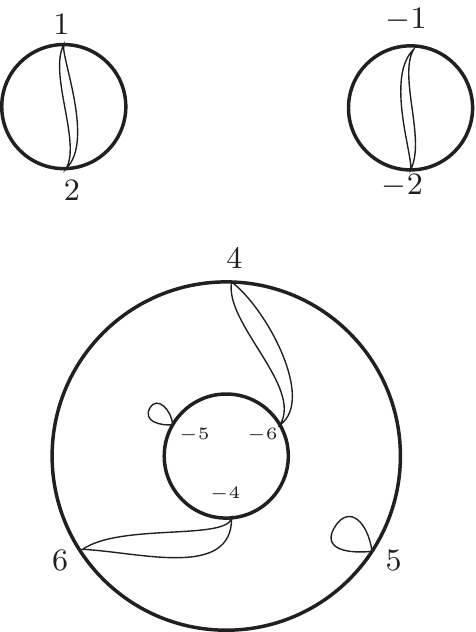} 

\bigskip

\caption{\label{fig3:the_second_type}\small On the left we have $\pi = (1, 2, 3) (-3, -2, -1)(4, -6)\ab(-4, 6)(5)(-5) \in S_{NC}^\delta(6, -6)$. $\pi^{-1}(1) = 3$, so $I_1 = \{1, 2, 3\}$ which does not meet any through blocks. So $\pi \in S_\textit{II}$. On the right we have squeezed off two circles $(1, 2, 3)$ and $(-3, -2, -1)$. Next,  $\pi^{-1}(1) = 3$ is removed from the block containing $1$ and $\pi(-1) = -3$ is removed from the block containing $-1$. }
\end{figure}

\begin{remark}\label{remark:T_III}
Given $\pi \in S_{NC}^\delta(n, -n)$ with $\pi^{-1}(1) \in [n]$. Let $j = \pi^{-1}(1)$ and $k$ be as in Proposition \ref{prop:left_or_right}, i.e. $k > 1$ is the smallest integer such that $\pi^{-1}(k) \in [-n]$. We suppose that we are in the case where $I_2 = \{j+1, \dots, n\}$ does not meet a through block of $\pi$, but $I_1$ does. So $2 \leq j \leq n - 1$. Thus $\pi_2 = \pi|_{I_2}$ is a non-crossing partition of $I_2$.  Next, $1$ and $j$ are in the same block of $\pi|_{I_1 \cup -I_1}$, as are $-1$ and $-j$. Let $\pi_1$ be the partition obtained from $\pi|_{I_1 \cup -I_1}$ obtained by removing $j$ from the block containing $1$, and $-j$ from the block containing $-1$. Since $\pi_1$ a restriction of a non-crossing annular permutation, it is itself a non-crossing annular permutation of $(j-1, -(j-1))$.
Since $\pi\delta$ is a pairing we have $\delta \pi \delta = \pi^{-1}$ and hence $\delta \pi_1 \delta = \pi_1^{-1}$. If $\pi_1 \delta$ had a singleton then so would $\pi$. Thus $\pi_1 \in S_{NC}(j-1, -(j-1))$. If we identify the points of $I_2$ with $[n - j]$ we have $\pi_2 \in NC(n-j)$. Hence we get a pair $(\pi_1, \pi_2) \in S_{NC}(j-1, -(j-1)) \times NC(n-j)$. We cannot have $j=1$ or $2$ as if either occurred $\pi$ could not have any through blocks. So $3 \leq j \leq n$. 
\end{remark}

\begin{theorem}
The map in Remark \ref{remark:T_III} is a bijection from $T_{\mathit{III}}$ to $S_{\mathit{III}}$.
\end{theorem}

\begin{figure}\noindent
\includegraphics{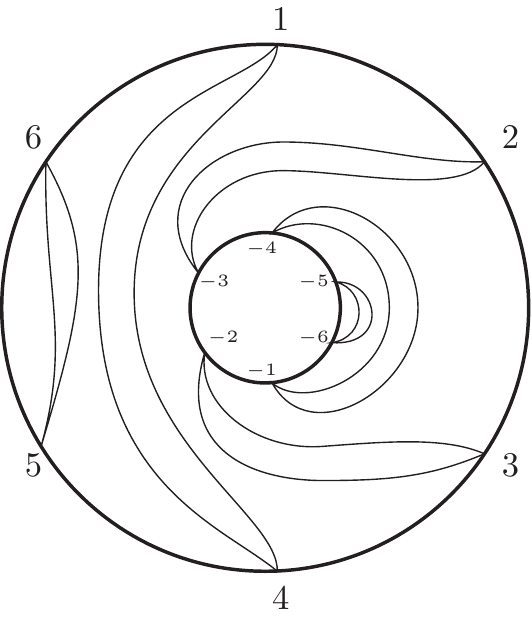} \hfill
\includegraphics{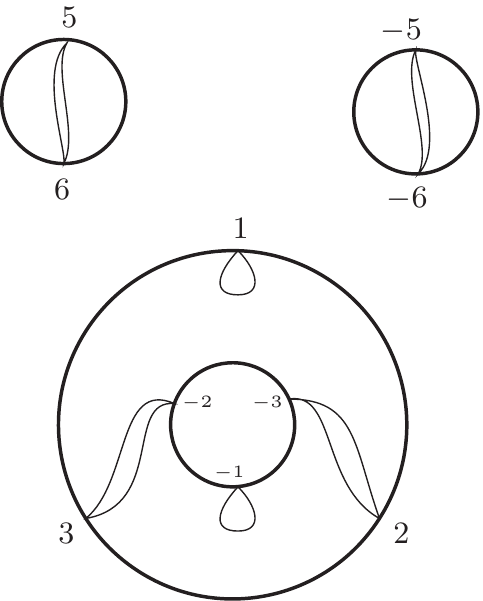} 

\bigskip

\caption{\label{fig3:the_third_type}\small On the left we have $\pi = (1, 4)(-1, -4)(2, -3)(3, -2)\ab(5, 6)(-5, -6) \in S_{NC}^\delta(6, -6)$. In this example $\pi^{-1}(1) = 4$ and $I_1 = \{1, 2, 3, 4\}$ does meet a through block. So $\pi \in S_\textit{III}$.  On the right we have $\pi_1 = (1)(-1)(2, -3)(-2, 3) \in S_{NC}(3, -3)$ and $\pi_2 = (5, 5) \in NC(2)$. }
\end{figure}

\begin{proof}
We need to give the inverse map, see Figure \ref{fig3:the_third_type}. Let $(\pi_1, \pi_2) \in S_{NC}(j-1, -(j-1)) \times NC(n-j)$ be given. We identify the points of $I_2 = \{ j+1, \dots, n\}$ with $[n - j]$ we have $\pi_2 \in NC(I_2)$. We add $j$ to block of $\pi_1$ containing $1$ and $-j$ to the block of $\pi_1$ containing $-1$, and denote this non-crossing partition $\tilde\pi_1$ i.e. $\tilde\pi_1 = \pi_1 (j, \pi_1^{-1}(1)) (-j, \pi_1(-1))$. Set $\pi = \tilde\pi_1  \pi_2\delta \pi_2^{-1}\delta$. Because we have inserted $\pi_2$ into the gap between $j+1$ and $n$, we get a non-crossing annular permutation. Hence $\pi \in S_{NC}(n, -n)$. Since $(\pi_2\delta \pi_2^{-1}\delta) \delta = \pi_2 \delta \pi_2^{-1}$ is a pairing we have that $\pi \delta$ is a pairing and thus $\pi \in S_{NC}^\delta(n, -n)$.
\end{proof}

\section{the recursion formula for $|S_{NC}^\delta(n, -n)|$}\label{sec:the_recursion}

Let $m_k = |NC(k)|$ and $m_k' = |S_{NC}^\delta(k, -k)|$. We adopt the convention that there is one partition of the empty set, so $m_0 = 1$. Also $S_{NC}^\delta(1, -1)$ is empty as there must be a through block, which has to be $(1, -1)$, but then $(1,-1)\delta$ is not a pairing, so $\ols{m}'_1 = 0$. 

\begin{theorem}\label{thm:the_recursion}
\[
\ols{m}'_n = (n-1) m_{n-1} +  \sum_{k=1}^{n}\{ m_{k-1}\,  \ols{m}'_{n-k} + \ols{m}'_{k-1} m_{n-k}\}.
\]

\end{theorem}

\begin{proof} 
When $n = 1$, $\ols{m}'_1 = 0$. On the other hand $(n-1)m_{n-1} = 0$ and the sum is empty; thus the identity holds for $n = 1$. Now fix $n> 1$. The cardinality of $T_I$ is $(n-1) m_{n-1}$. From Definition \ref{def:the_division} $(e)$, the cardinality of $T_\mathrm{II}$ is $\sum_{k = 1}^{n-2} m_{k-1}\, \ols{m}'_{n-k}$, keeping in mind that the term for $k = n-1$ is $0$.  The cardinality of $T_\mathrm{III}$ is $\sum_{k=2}^{n} \ols{m}'_{k-1} m_{n-k} = \sum_{k=3}^n \ols{m}'_{k-1} m_{n-k} = \sum_{k = 1}^{n-2} m_{k-1}\, \ols{m}'_{n-k}$. Adding these up we get the claim in the theorem.
\end{proof}

\begin{theorem}\label{thm:special_moment_generating_function}
Let $M(z) = \sum_{n=0}^\infty m_n z^n$ and $m(z) = \sum_{n=1}^\infty \ols{m}'_n z^n$. Then 
\[
m(z) = \frac{z^2M'(z)}{1 - 2 z M(z)} = \frac{1 - 2 z - \sqrt{1 - 4z}}{2(1 - 4 z)}.
\] 
\end{theorem}

\begin{proof}
We begin by treating $m$ and $M$ as formal power series, so we may differentiate under the  summation sign to obtain
\begin{align*}
\sum_{n=1}^\infty (n-1) m_{n-1} z^n 
&=
z^2 \sum_{n=1}^\infty m_{n-1} (n -1) z^{n-2} \\
&=
z^2 \frac{d}{dz} \sum_{n=1}^\infty m_{n-1} z^{n-1}
=
z^2 M'(z).
\end{align*}
Next, observe that
\begin{align*}\lefteqn{
\sum_{n=1}^\infty \sum_{k=1}^{n-1} m_{k-1} \ols{m}'_{n-k} z^n 
=
 z \sum_{k=1}^\infty \sum_{n=k+1}^\infty m_{k-1} z^{k-1} \ols{m}'_{n-k} z^{n-k} }\\
& =
z \sum_{k=1}^\infty m_{k-1} z^{k-1} 
  \sum_{n=k+1}^\infty \ols{m}'_{n-k} z^{n-k} 
=
z M(z)  \sum_{n=1}^\infty \ols{m}'_{n} z^{n} = z m(z) M(z). 
\end{align*}
Thus the recursion equation in Theorem \ref{thm:the_recursion} gives us
\[
m(z) = z^2 M'(z)  + 2 z m(z) M(z). 
\]
This gives us the first claimed equality. Now we have $M(z) = \frac{1 - \sqrt{ 1 - 4 z}}{2 z}$ has a radius of convergence of $1/4$ so $M$ is analytic on $D$, the open disc with centre $0$ and radius $1/4$. As $1 - 2 z M(z) = \sqrt{1 - 4 z}$ we have that $1 - 2 z M(z) \not= 0$ on $D$. Hence $m$ is analytic on $D$. Finally we have $z^2 M'(z) = \frac{ 1 - 2 z - \sqrt{1 - 4 z}}{2\sqrt{1 - 4 z}}$, this gives us the second claimed equality.
\end{proof}

\begin{corollary} \label{cor:exact_formula}
For $n \geq 1$
\[
\ols{m}'_n = 4^{n-1} - \frac{1}{2} \binom{2n}{n}. 
\]
\end{corollary}

\begin{proof}
We have
\[
\sum_{n=1}^\infty 4^{n-1} z^n = \frac{ z}{1 - 4 z}
\mbox{\ and\ }
\sum_{n=1}^\infty \binom{2n}{n} z^2 = \frac{1 - \sqrt{1 - 4 z}}{\sqrt{1 - 4 z}}.
\]
So 
\[
\sum_{n=1}^\infty \Big[ 4^{n-1} - \frac{1}{2} \binom{2n}{n} \Big] z^n
= 
\frac{z}{1 - 4z} - \frac{1 - \sqrt{1 - 4 z}}{2 \sqrt{1 - 4 z}}  
= 
m(z).
\]
\end{proof}

\section{the recursion formula for the infinitesimal\\ moments of a real wishart matrix }\label{sec:the_mp_recursion}

Let recall the free cumulants of the Marchenko-Pastur law with parameter $c$: $\kappa_n = c$ for $n \geq 1$. This means that the moments are given by
\[
m_n = \sum_{\pi \in NC(n)} c^{\#(\pi)}.
\]
From this one gets the following recursion for the moments:
\begin{equation}\label{eq:mp_recurrence}
m_n = (c-1) m_{n-1} + \sum_{k=1}^{n} m_{k-1} m_{n-k}.
\end{equation}
Our goal in this section is to establish the recurrence
\begin{equation}\label{eq:inf_mp_recurrence}
\ols{m}'_n  = (n-1)m_{n-1} + (c-1) \ols{m}'_{n-1} + \sum_{k=1}^{n} \{\ols{m}'_{k-1} m_{n-k} + m_{k-1} \ols{m}'_{n-k} \}.
\end{equation}
for the infinitesimal moments $\{ \ols{m}'_n \}_{n = 1}^\infty$. 
Note that the sequences $\{ m_n \}_{n = 1}^\infty$ and $\{ \ols{m}'_n \}_{n = 1}^\infty$ are a sequences of polynomials in $c$ with integer coefficients, see Proposition \ref{prop:counting_diagrams}. When $c = 1$, equation (\ref{eq:mp_recurrence}) is the usual recurrence for the Catalan numbers. When $c=1$, equation (\ref{eq:inf_mp_recurrence}) gives  the recurrence in Theorem \ref{thm:the_recursion}.  As the method for getting the second recurrence is the annular version of the recurrence for the Catalan numbers, we shall review the method here. 

For the purposes of this section we shall regard $NC(n)$ as a subset of the symmetric group $S_n$, where the cycles of the permutation are the blocks of the partition. As usual we embed $S_n$ into $S_{n+1}$ by making it act trivially on $\{ n + 1\}$. From the point of view of partitions this means adding to a partition of $[n]$ the one element block $(n + 1)$ to get a partition of $[n + 1]$.

If we let $\tau_n = (1 , n)$ and $\sigma \in S_{n-1}$ then $ \tau_{n}\sigma \in S_{n}$ is  the permutation whose cycles are unchanged except the one containing $1$. If this cycle was $(i_1, \dots, i_k)$ with $i_1 = 1$, the it becomes $(n, i_1, \dots, i_k)$ in $\tau_n\sigma$. If $\sigma$ was non-crossing then $\tau_n \sigma$ is also non-crossing; and $(\tau_n \sigma)^{-1}(1) = n$. 

Conversely if $\sigma \in NC(n)$ and $\sigma^{-1}(1) = n$ then $\tau_n \sigma$ leaves $n$ fixed and so $\tau_n \sigma \in S_{n-1}$; moreover $\tau_n \sigma \in NC(n-1)$ because we cannot introduce  crossing by removing a block.

\begin{lemma}\label{lemma:removing_a_point}
The map $\sigma \mapsto \tau_n \sigma$ maps $NC(n-1)$ bijectively onto $\{ \pi \in NC(n) \mid 1$ and $n$ are in the same block$\}$.
\end{lemma}

\begin{remark}
Let $NC(n)_k$ denote the set of non-crossing partitions of $[n]$ such that $k$ is the largest element in the block containing $1$. Then we have $NC(n) = \ds\mathop{\cup}_{k=1}^n NC(n)_k$. In addition $NC(n)_k \ni \sigma \mapsto \sigma_1 \times \sigma_2 \in NC(k-1) \times NC(n-k)$ is a bijection where $\sigma_1$ is the element of $NC(k-1)$ produced by restricting $\sigma$ to $[k]$ and then applying Lemma \ref{lemma:removing_a_point}. We get $\sigma_2$ by restricting $\sigma$ to $[k+1, n]$.

For the Marchenko-Pastur law with parameter $c$ we have 
\begin{equation}\label{eq:mp_moments}
m_n = \sum_{\pi \in NC(n)} c^{\#(\pi)}.
\end{equation} 
Thus by Lemma \ref{lemma:removing_a_point}
\begin{align*}
m_n &= \sum_{\pi \in NC(n)} c^{\#(\pi)} = \sum_{k=1}^n \sum_{\pi \in NC(n)_k}
c^{\#(\pi)} \\
& = c \sum_{\pi \in NC(n-1)} c^{\#(\pi)} + 
\sum_{k=2}^n \sum_{\pi_1 \in NC(k-1)} \sum_{\pi_2 \in NC(n-k)}  c^{\#(\pi_1)} c^{\#(\pi_2)}      \\
& =
c m_{n-1} + \sum_{k=2}^n m_{k-1} m_{n-k} 
=
(c - 1) m_{n-1} + \sum_{k=1}^n m_{k-1} m_{n-k}.
\end{align*}
This gives us the recurrence for the moments of the Marchenko-Pastur law: equation (\ref{eq:mp_recurrence}).
\end{remark}

Now let us turn to the infinitesimal law of a real Wishart matrix. We have that using the notation of Definition \ref{def:the_division}
\begin{align*}
\ols{m}'_n
&=
\sum_{\pi \in S_{NC}^\delta(n, -n)} c^{\#(\pi)/2} \\
&=
\sum_{\pi \in S_I} c^{\#(\pi)/2} +
\sum_{\pi \in S_{\mathit{II}}} c^{\#(\pi)/2} +
\sum_{\pi \in S_{\mathit{III}}} c^{\#(\pi)/2}.
\end{align*}
Now let us separately find an expression for each term.

\begin{lemma}\label{lemma:s_I_term}
\[
\sum_{\pi \in S_I} c^{\#(\pi)/2} 
=
(n-1) m_{n-1}.
\]
\end{lemma}

\begin{proof}
By Theorem \ref{thm:T_I_bijection}
\[
\sum_{\pi \in S_I} c^{\#(\pi)/2} 
=
(n-1) \sum_{\pi_1 \in NC(n-1)} c^{\#(\pi_1)}
\]
because when we go from $\pi \in S_I$ we let $k = -\pi^{-1}(1)$ and we get $\pi_1$ by restricting $\pi$ to 
\[
(1, 2, 3, \dots, k-1, -n, -(n-1), \dots, -(k+1), -k).
\]
and then we get $\pi \in NC(\{ 1, 2, 3, \dots, k-1, -n, -(n-1), \dots, -(k+1), -k\})$ by removing $-k$ from the cycle containing $1$. Since there are exactly the same number of cycles of $\pi$ in the other cycle of $\hg$
\[
(k, k+1, \dots, n-1, n, -(k-1), \dots, -3, -2, -1)
\]
we have $\#(\pi) = 2 \#(\pi_1)$. Now sum over $k$, noting that we cannot have $k= 1$,  this gives the factor of $(n-1)$. 
\end{proof}

\begin{lemma}\label{lemma:s_II_term}
\[
\sum_{\pi \in S_\mathit{II}} c^{\#(\pi)/2} 
=
(c - 1) \ols{m}'_{n-1} + \sum_{j=1}^{n - 2} m_{j-1} \ols{m}'_{n - j}.
\]
\end{lemma}

\begin{proof}
Let $j = \pi^{-1}(1)$. We will break the proof into two parts; the first part is when $j = 1$ and the second part is when $2 \leq j \leq n -2$. As noted in Remark \ref{remark:T_II}, the largest $j$ can be is $n - 2$. 

When $j = 1$ then $\pi(1) = 1$. Thus $\pi_1$ is the empty partition with $0$ blocks, so $c^{\#(\pi_1)} = 1$. Thus $\#(\pi_2) = \#(\pi) -2$. Hence the contribution for $j = 1$ is $c \, \ols{m}'_{n-1}$. 

When $j \geq 2$ we have $\#(\pi) = 2 \#(\pi_1) + \#(\pi_2)$. Thus the contribution for $j \geq 2$ is $m_{j-1} \ols{m}'_{n-j}$. Summing we have, as $m_0 =1$,
\begin{align*}
\sum_{j=1}^{n-2} \mathop{\sum_{\pi \in S_{\mathit{II}}}}_{\pi^{-1}(1) = j} c^{\#(\pi)/2}
&=
c \ols{m}'_{n-1} + \sum_{j=2}^{n-2} m_{j-1} \ols{m}'_{n-j} \\
&=
(c - 1) \ols{m}'_{n-1} + \sum_{j=1}^{n-2} m_{j-1} \ols{m}'_{n-j}.
\end{align*}
\end{proof}

\begin{lemma}\label{lemma:s_III_term}
\[
\sum_{\pi \in S_\mathit{III}} c^{\#(\pi)/2} 
=
\sum_{j=1}^{n - 2} m_{j-1} \ols{m}'_{n - j}.
\]
\end{lemma}

\begin{proof}
Let $j = \pi^{-1}(1)$. By Remark \ref{remark:T_III} we have $3 \leq j \leq n$. Thus
\[
\sum_{\pi \in S_{\mathit{III}}} c^{\#(\pi)/2}
=
\sum_{j=3}^n \ols{m}'_{j-1} m_{n - j}
= \sum_{j=1}^{n-2} m_{j-1} \ols{m}'_{n - j}
\]
\end{proof}

\begin{theorem}\label{thm:c_recursion}
Let $\ols{m}'_n = \ds\sum_{\mathclap{\pi \in S_{NC}^\delta(n, -n)}} c^{\#(\pi)/2}$. Then $\ols{m}'_1 = 0$ and for $n \geq 2$ we have
\[
\ols{m}'_n = (n - 1) m_{n-1} + (c - 1) \ols{m}'_{n-1} + 2 \sum_{k=1}^{n -2}
m_{k-1} \ols{m}'_{n- k}.
\]
\end{theorem}

\begin{proof}
For $n = 1$, $\ols{m}'_1 = 0$ because $S_{NC}^\delta(1, -1)$ is empty. For $n \geq 2$ we have
\begin{align*}
\sum_{\pi \in S_{NC}^\delta(n, -n)} c^{\#(\pi)/2}
&=
\sum_{\pi \in S_{\mathit{I}}} c^{\#(\pi)/2}
+
\sum_{\pi \in S_{\mathit{II}}} c^{\#(\pi)/2}
+
\sum_{\pi \in S_{\mathit{II}}} c^{\#(\pi)/2} \\
& =
(n - 1) m_{n-1} + (c - 1) \ols{m}'_{n-1} + 2 \sum_{k=1}^{n-2} m_{k-1} \ols{m}'_{n - k}.
\end{align*}
\end{proof}

Let $a = (1 - \sqrt{c})^2$, $b = (1 + \sqrt{c})^2$, and $\{m_n\}_{n=1}^\infty$ be the moments of the Marchenko-Pastur law with parameter $c$, (see equation \ref{eq:mp_moments}). Recall that the moment generating function $M(z) = 1 + \sum_{n=1}^\infty m_n z^n$ is given by
\[
M(z) = \frac{1 - (1 + c)z - \sqrt{(1 - a z)(1 - b z)}}{2 z}
\]

\begin{theorem}\label{thm:generating_function}
The generating function $\cm(z) = \sum_{n=2}^\infty \ols{m}'_n z^n$ is given by
\[
\cm(z) = \frac{ z^2 M'(z)}{1 + z (1 - c) - 2 z M(z)}.
\]
\end{theorem}

\begin{remark}
When $c = 1$ we get the formula of Theorem \ref{thm:special_moment_generating_function}. 

\end{remark}

\begin{proof}
\begin{align*}
\cm(z) 
& =
\sum_{n = 2}^\infty
(n - 1) m_{n-1} z^n
+
(c -1) \sum_{n=2}^\infty \ols{m}'_{n-1} z^n  \\
& \qquad \mbox{}+ 
2 \sum_{n=2}^\infty
  \sum_{k=1}^{n-2} m_{k-1} \ols{m}'_{n-k} z^n
\end{align*}
and we shall do each sum separately. 

As in the proof of Theorem \ref{thm:special_moment_generating_function} we have $\sum_{n = 2}^\infty
(n - 1) m_{n-1} z^n = z^2 M'(z)$. Also $\sum_{n=2}^\infty \ols{m}'_{n-1} z^n = z \cm(z)$. Finally
\begin{multline*}
\sum_{n=2}^\infty
  \sum_{k=1}^{n-2} m_{k-1} \ols{m}'_{n-k} z^n
=
z \sum_{k=1}^\infty
  \sum_{n=k+2}^\infty m_{k-1} \ols{m}'_{n-k} z^{n-1}  \\
=
z \sum_{k=1}^\infty m_{k-1} z^{k-1}
\sum_{n=k+2}^\infty  \ols{m}'_{n-k} z^{n-k}
= z \cm(z) M(z). 
\end{multline*}
Thus
\[
\cm(z) = z^2 M'(z) + (c - 1) z \cm(z) + 2 z \cm(z) M(z).
\]
\end{proof}

Let $P(z) = (z - a)(z - b)$ and $G$ be the Cauchy transform  of the Marchenko-Pastur law with parameter $c$. Recall that
\[
G(z) = \frac{1}{z} + \sum_{n=1}^\infty \frac{m_n}{z^{n + 1}} =
\frac{z + 1 - c - \sqrt{P(z)}}{2 z}.
\]
Let $g(z) = \frac{1}{z} \cm\big( \frac{1}{z} \big)$ be the (infinitesimal) Cauchy transform of the moment sequence $\{ \ols{m}'_n\}_{n = 1}^\infty$, see e.g. \cite[\S 2]{m}. The next theorem shows that the second term of Equation (\ref{eq:two_terms}) gives the same distribution as in Dumitriu and Edelman \cite{dumitriu_edelman_2006}.
We choose the branch of $\sqrt{P(z)}$ as in \cite[Ex. 3.6]{m}.

\begin{theorem}\label{thm:second_cauchy_transform}
\begin{equation}\label{eq:infinitesimal_cauchy_transform}
g(z) 
= 
\frac{zG(z) -1}{P(z)}
= 
\frac{1}{2}
\Big\{\ 
\frac{1}{2}\Big\{ \frac{1}{z - a} + \frac{1}{z -b}\Big\}
-
\frac{1}{\sqrt{(z - a) (z - b)}} \ \Big\}. 
\end{equation}
\end{theorem}

\begin{proof}
Since $M(z^{-1}) = z G(z)$ we have $-z^{-2} M'(z^{-1}) = G(z) + z G'(z)$. It is routine to check that $G(z) + z G'(z) = \ds\frac{1 - z G(z)}{\sqrt{P(z)}}$. Hence
\begin{align*}
g(z) = \frac{1}{z} \cm\Big( \frac{1}{z} \Big)
=&\ 
\frac{1}{z} \frac{ z^{-2} M'(z^{-1})}{1 + z^{-1}(1 - c) - 2 z^{-1} M(z^{-1})} \\
& =
\frac{z^{-2} M'(z^{-1})}{z + (1 - c) - 2 z G(z)}
=
\frac{z^{-2} M'(z^{-1})}{\sqrt{P(z)}} \\
=
\frac{zG(z) - 1}{P(z)}
&=
\frac{1}{2}
\Big\{\ 
\frac{1}{2}\Big\{ \frac{1}{z - a} + \frac{1}{z -b}\Big\}
-
\frac{1}{\sqrt{(z - a) (z - b)}} \ \Big\}
\end{align*}
\end{proof}

\begin{remark}
If we let $\nu_1$ be the signed measure:
\begin{equation}
d\nu_1(x) = -c'
\begin{cases}
\phantom{\frac{1}{2}}\delta_0 -  \frac{x + 1 -c}
{2 \strut\pi x \sqrt{(b - x)(x - a)}}\,dx & c < 1 \\
\frac{1}{2} \delta_0 -  \frac{1}{2\pi\sqrt{x(4 - x)}}\,dx & c = 1 \\
\phantom{\frac{1}{2} \delta_0}
-\frac{x + 1 -c}
{2 \strut\pi x \sqrt{(b - x)(x - a)}}\,dx  & c > 1 \\
\end{cases}
\end{equation}
where $a = (1 -  \sqrt c)^2$ and $b = (1 + \sqrt c)^2$ and 
\begin{equation}\label{eq:infinitesimal_law}
d\nu_2(t) 
=
\frac{1}{2}\bigg( \frac{1}{2}(\delta_a + \delta_b) - \frac{1}{\pi \sqrt{(b-t)(t - a)}}\bigg),
\end{equation}
and $\mu' = \nu_1 + \nu_2$ then $m'_n = \int t^n \,d\mu'(t)$. Combining the conclusion of Theorem \ref{thm:second_cauchy_transform} with \cite[Thm. 31]{m} we have the infinitesimal Cauchy transform for $\{ m'_n\}_n$ is 
\begin{align*}\lefteqn{
g(z) = \sum_{n=0}^\infty \frac{m'_n}{z^{n+1}} } \\
& =
\frac{-c'}{z \sqrt{P(z)}}
\frac{(1 - c)^2 - (1 + c)z - (1 - c)\sqrt{P(z)}}{\sqrt{P(z)} + z -1 + c} \\
& \qquad\mbox{} + 
\frac{1}{2}
\Big\{\ 
\frac{1}{2}\Big\{ \frac{1}{z - a} + \frac{1}{z -b}\Big\}
-
\frac{1}{\sqrt{(z - a) (z - b)}} \ \Big\}.
\end{align*}

\section{infinitesimal  \textsc{{r}}-transform}
\label{sec:infinitesimal_r-transform}

In Arizmendi, Garza-Vargas, and Perales \cite[Example 5.8]{agp} gave the infinitesimal $R$-transform of negative of the measure in Equation (\ref{eq:infinitesimal_law}). Let us recall that given two moment sequences $\{ m_n \}_{n \geq 1}$ and $\{ m'_n \}_{n \geq 1}$ the infinitesimal cumulants defined by F\'evrier and Nica \cite[\S 1.2]{fn} are given by formally differentiating the moment cumulant relation
\begin{equation}\label{eq:moment_cumulant}
m_n = \sum_{\pi \in NC(n)} \kappa_\pi
\quad \mathrm{\ to\ get\ }\quad
m'_n = \sum_{\pi \in NC(n)} \partial \kappa_\pi
\end{equation}
where $\partial\kappa_\pi$ is computed using the Leibnitz rule:
\[
\partial\kappa_\pi = \sum_{\pi \in NC(n)} \sum_{V \in \pi}
\kappa'_{|V|} \mathop{\prod_{W \in \pi}}_{W \not = V} \kappa_{|V|}.
\]
There is an equivalent operator valued version of this relation as follows. Let $M_n = \begin{bmatrix} m_n & m'_n \\ 0 & m_n \end{bmatrix}$ and  $K_n = \begin{bmatrix} \kappa_n & \kappa'_n \\ 0 & \kappa_n \end{bmatrix}$. Then Equation (\ref{eq:moment_cumulant}) becomes
$\ds
M_n = \sum_{\pi \in NC(n)} K_\pi$
where
$\ds K_\pi = \prod_{V \in \pi} K_{|V|}.
$
Now let $Z = \begin{bmatrix} z & w \\ 0 & z \end{bmatrix}$, $G(Z)  = \sum_{n=0}^\infty M_n Z^{-(n+1)}$ and $R(Z) = \sum_{n=1}^\infty K_n Z^{n-1}$. Expanding out the series we have 
\[
G(Z) = \begin{bmatrix} G(z) & w G'(z) + g(z) \\ 0 & G(z) \end{bmatrix} 
\mathrm{\ and \ }
R(Z) = \begin{bmatrix} R(z) & w R'(z) + r(z) \\ 0 & R(z) \end{bmatrix}. 
\]
Then, as all these $2 \times 2 $ matrices commute, Equation (\ref{eq:moment_cumulant}) implies the Voiculescu relation 
\[
Z^{-1} + R(G(Z)) = Z
\]
which then implies that $r(z) = -g(K(z)) K'(z)$ where $K = G^{\langle -1 \rangle}$ denotes the compositional inverse of $G$, $K'$ is the derivative of $K$ with respect to $z$,  $g(z) = \sum_{n=1}^\infty m'_n z^{-(n +1)}$, and $
r(z) = \sum_{n=1}^\infty \kappa'_n z^{n-1}$, is the infinitesimal $R$-transform.
See Theorem 2 from \cite{m}. 

Here we will briefly show that one can also do a direct computation of  $r$ using $r(z) = -g(K(z)) K'(z)$. 

We shall let $\ds
G(z) = \frac{z + 1 - c \sqrt{(z - a)(z - b)}}{2z}$ be the Cauchy transform of the Marchenko-Pastur law, $g$ be the Cauchy transform in Equation (\ref{eq:infinitesimal_cauchy_transform}). Then, one lets $c_1 = (1 - \sqrt{c})^{-1}$, $c_2 = (1 + \sqrt c)^{-1}$, and then checks the following steps: 
\begin{itemize}
\item
$\ds
K(z) = \frac{1}{z} + \frac{c}{1 - z}
$

\item
$\ds
\sqrt{P(K(z))} = K(z) + 1 - c - 2 z K(z)
= (c-1) \frac{(z - c_1)(z - c_2)}{z(z - 1)}$

\item
$\ds
K'(z) = (c-1) \frac{(z - c_1)(z - c_2)}{z^2(z - 1)^2}
$

\item
$\ds
g(K(z)) K'(z) = (c -1)^{-1} \frac{c z}{1 - z} \ \frac{1}{(z - c_1)(z - c_2)}$. 
\end{itemize}
Upon doing a partial fraction expansion we then get
\[
r(z) = \frac{1}{2} \Big\{ \frac{1}{c_1 - z} + \frac{1}{c_2 - z}  - \frac{2}{1 - z}\big\}\ \mathrm{which\ gives}\ 
\kappa'_n = \frac{1}{2}\{ c_1^{-n} + c_2^{-n} - 2\}.
\]
This produces the following table.

\begin{center}
$
\begin{array}{c|cl}
n  & \kappa_n & \kappa'_n \\ \hline
 1 & c & 0 \\
 2 & c & c \\
 3 & c & 3 c \\
 4 & c & c^2+6 c \\
 5 & c & 5 c^2+10 c \\
 6 & c & c^3+15 c^2+15 c \\
 7 & c & 7 c^3+35 c^2+21 c \\
 8 & c & c^4+28 c^3+70 c^2+28 c \\
\end{array}
$\end{center}
This enables to come full circle and check agreement with Figure \ref{table:the_first_six_moments}. For example Equation (\ref{eq:moment_cumulant}) gives us that $m_3 = \kappa_3 + 3 \kappa_1 \kappa_2 + \kappa_1^3$ and $m'_3 = \kappa'_3 + 3 \kappa'_1 \kappa_2 + 3 \kappa_1 \kappa_2 + 3 \kappa_1^2 \kappa'_1$. Using that $\kappa_n = c$ for all $n$ and $\kappa'_1 = 0$ we have $m'_3 = \kappa'_3 + 2 \kappa_1 \kappa'_2 = 3 c + 3 c^3$ as claimed in Figure \ref{table:the_first_six_moments}.

\section{Infinitesimal Distributions of\\ Some Orthogonal Polynomials}
\label{sec:orthogonal_polynomials}

Given one of the classical ensembles of orthogonal polynomials (see e.g. \cite[\S 7.6]{sz}), we can create a sequence of probability measures by putting a mass of $1/n$ at each of the $n$ zeros of $p_n$, the polynomial in the ensemble of degree $n$. In the case of the Hermite and Laguerre polynomials the corresponding measures converge to the semi-circle law and the Marchenko-Pastur law respectively. Moreover the convergence is such that there are infinitesimal laws as well. In the case of the Hermite polynomials we get the negative of the infinitesimal law of the GOE: $\mu' = \nu_1 - \nu_2$ where $\nu_1 = \frac{1}{2}( \delta_{-2} + \delta_{2})$ is the Bernoulli law and $d\nu_2(t) = $ on the interval $[-2, 2]$ is the arcsine law, see Arizmendi, Garza-Vargas, and Perales in \cite[Example 5.6]{agp}. In the case of the Laguerre polynomials it was shown by Arizmendi, Garza-Vargas, and Perales in \cite{agp} that one gets the negative of the measure in Equation (\ref{eq:infinitesimal_law}).  Indeed, from \cite[Example 5.8]{agp} we see that the negative of Cauchy transform of the infinitesimal  distribution of the zeros of the Laguerre polynomials is
\begin{equation}
g_{{\textsc{\tiny mp}}}(z) = \frac{zG(z) - 1}{P(z)}
=
\frac{1}{2}
\Big\{\ 
\frac{1}{2}\Big\{ \frac{1}{z - a} + \frac{1}{z -b}\Big\}
-
\frac{1}{\sqrt{(z - a) (z - b)}} \ \Big\}.
\end{equation}
From \cite[Remark 28]{m} we see that the  Cauchy transform of the infinitesimal measure obtained from the GOE is 
\[
g_{{\textsc{\tiny sc}}}(z) = \frac{1}{2} \Big\{ \frac{1}{2} \Big( \frac{1}{z - 2} + \frac{1}{z + 2}\Big) - \frac{1}{ \sqrt{z^2 - 4}}\Big\}. 
\]
If we let $w = ( z - (1 + c))/\sqrt{c}$, then we have $g_{{\textsc{\tiny mp}}}(z) = g_{{\textsc{\tiny sc}}}(w)/\sqrt{c}$. Thus up to a change of variable and rescaling these distributions are the same.
\end{remark}

\thebottomline

\begin{thebibliography}{X}

\bibitem{agp} O.~Arizmendi, J.~Garza-Vargas, and
  D.~Perales, Finite Free Cumulants: Multiplicative
  Convolutions, Genus Expansion and Infinitesimal
  Distributions, \textit{arXiv}:2108.08489

\bibitem{bs} S.~Belinschi and D.~Shlyakhtenko, Free
  Probability of Type B: Analytic Interpretation and
  Applications, \textit{Amer. J. Math.} \textbf{134} (2012),
  193-234.
  
\bibitem{bgn} P.~Biane, F.~Goodman, and A.~Nica,
  Non-crossing Cumulants of Type B,
  \textit{Trans. Amer. Math. Soc.} \textbf{355} (2003),
  2263-2303.

\bibitem{bcgls} G.~Borot, S.~Charbonnier, E.~Garcia-Failde, F.~Leid, S.~Shadrin
Analytic theory of higher order free cumulants, \textit{arXiv}:2112.12184.

\bibitem{cmss} B. Collins, J. A. Mingo,
  P. \'Sniady, and R. Speicher, \newblock{Second Order
    Freeness and Fluctuations of Random Matrices:
    III. Higher Order Freeness and Free Cumulants},
  \newblock{\em Documenta Math.}, \textbf{12} (2007), 1-70.
  
  
\bibitem{dumitriu_edelman_2006} I.~Dumitriu and A.~Edelman.
  \newblock Global spectrum fluctuations for the
  $\beta$-{H}ermite and $\beta$-{L}aguerre ensembles via
  matrix models.  \newblock {\em J.~Math. Phy.}, 47(063302),
  2006, 36pp.


\bibitem{f} V.~F\'eray, On Complete Functions in
  Jucys-Murphy Elements, \textit{Ann. Comb.} \textbf{16}
  (2012), 677-707.

\bibitem{fn} M.~F\'evrier and A.~Nica, Infinitesimal
  non-crossing cumulants and free probability of type B,
  \textit{J.~Funct.~Anal.} \textbf{258} (2010), 2983-3023.
  
\bibitem{glm} P.~Graczyk, G.~Letac, and H.~Massam, The 
   Hyperoctahedral group, symmetric group representations
   and the moments of the real Wishart distribution, 
   \textit{J.~Theoret.~Probab.} \textbf{18} (2005), 
   1-42.

\bibitem{gkp} R.~Graham, D.~Knuth, and O.~Patashnik,
  \textit{Concrete Mathematics}, $2^\textit{nd}$ ed.,
  Addison-Wesley, Reading, MA, 1994.

\bibitem{h} G.~'t Hooft, A planar diagram theory theory
for strong interactions, \textit{Nuclear Phy. B} \textbf{72}
(1974), 461-473.

\bibitem{j} K.~Johansson, On fluctuations of random
  Hermitian matrices, \textit{Duke Math.~J.} \textbf{91}
  (1998) 151-203.

\bibitem{kms} T. Kusalik, J. A. Mingo, and R. Speicher,
  Orthogonal polynomials and fluctuations of random
  matrices, {\em J. Reine Angew.  Math.}, {\bf 604} (2007),
  1 - 46.

\bibitem{msv} D.~Merlini, R.~Sprugnoli, and M.~C.~Verri, The
  area determined by underdiagonal lattice paths, in
  \textit{Trees in Algebra and Programming—CAAP’96},
  H.~Kirchner, (ed.), Springer-Verlag, London (1996), 59-71.

\bibitem{m} J.~A.~Mingo, Non-crossing annular pairings and
  the infinitesimal distribution of the {\sc goe},
  \textit{J.~Lond.~Math.~Soc.} \textbf{100}, (2019),
  987-1012.


\bibitem{mn} J.~A.~Mingo and A.~Nica, Annular Non-crossing
  Permutations and Partitions, and Second Order Asymptotics
  for Random Matrices, \textit{IMRN} \textbf{2004}, no. 28,
  1413-1460.

\bibitem{mp} J.~A.~Mingo and M.~Popa, Real second order
  freeness and Haar orthogonal matrices,
  \textit{J. Math. Phy.} \textbf{54} (2013), 051701, 1-35.

  
\bibitem{ms} J.~A.~Mingo and R.~Speicher, \textit{Free
  Probability and Random Matrices}, Springer, New York,
  2017.

\bibitem{no} A.~Nica and I.~Oancea,
Posets of annular non-crossing partitions of types B and D, \textit{Discrete Math.} \textbf{309}  (2009), 1443-1466.

\bibitem{ns} A.~Nica and R.~Speicher, \textit{Lectures on
  the Combinatorics of Free Probability}, Cambridge
  University Press, Cambridge, 2006.

\bibitem{s} R.~Speicher, Multiplicative functions on the 
lattice of non-crossing partitions and free convolution, 
\textit{Math. Annalen} \textbf{298} (1994), 611-628.


\bibitem{r1} C. Emily I. Redelmeier, 
Genus expansion for real Wishart matrices,
\textit{ J. Theoret. Probab.}, \textbf{24} (2011), 1044--1062.


\bibitem{r2} C. E. I. Redelmeier, Real second-order
  freeness and the asymptotic real second-order freeness of
  several real matrix ensembles,
  \textit{Int. Math. Res. Not.}  \textbf{2014}, no. 12,
  pp. 3353-3395.

\bibitem{sl} N.~J.~A.~Sloan and the OEIS Foundation, 
\textit{The on-line encyclopedia of integer sequences}, 
\verb|http://oeis.org/|.
  
\bibitem{sh} D.~Shlyakhtenko, Free Probability of Type B and
  Asymptotics of Finite Rank Perturbations of Random
  Matrices, \textit{Indiana Univ. Math. J.} \textbf{67} (2018), 971-991.

\bibitem{sz} G.~Szego\"{o}, \textit{Orthogonal Polynomials}, 4$^{th}$ ed., Amer. Math. Soc., Providence R.~I., 1975.

\bibitem{v}
D.~Voiculescu, Limit laws for random matrices and free products, \textit{Invent. Math. } \textbf{104} (1991), 201-220.
\end{thebibliography}
\end{document}